\documentclass{svjour3}                     % onecolumn (standard format)
\usepackage[compress]{cite}
\usepackage{amsmath}
\usepackage{amssymb}
\usepackage{graphicx}
\usepackage{url}
\usepackage[colorlinks,linkcolor=blue,citecolor=red,urlcolor=red]{hyperref}
\usepackage{float}
\usepackage{colortbl,dcolumn}     %% 2šºšŠ?¡Àšª??
\usepackage{color}
\usepackage{bm}

  \numberwithin{figure}{section}
  \numberwithin{table}{section}
\numberwithin{remark}{section}
\numberwithin{lemma}{section}
\numberwithin{theorem}{section}
\numberwithin{equation}{section}
\numberwithin{proposition}{section}
\numberwithin{lemma}{section}

\numberwithin{assumption}{section}
\usepackage{graphicx}
\makeatletter
\renewcommand\normalsize{%
   \@setfontsize\normalsize\@xpt\@xiipt
   \abovedisplayskip 7\p@ \@plus2\p@ \@minus5\p@
   \abovedisplayshortskip 6\z@ \@plus3\p@
   \belowdisplayshortskip 6\p@ \@plus3\p@ \@minus3\p@
   \belowdisplayskip 7\p@ \@plus2\p@ \@minus5\p@
   \let\@listi\@listI}
\makeatother
\usepackage{multirow,multicol}

\newcommand{\blue}[1]{{\color{blue}#1}}

\newcommand{\mg}[1]{{\color{magenta}#1}}

\def\C{\mathbb C}
\DeclareMathOperator\diag{{diag}}
\DeclareMathOperator\re{{Re}}
\DeclareMathOperator\im{{Im}}

\def\dg{diagonalization }
\newcommand{\T}{\mathsf{T}}

\newcommand{\CB}{{\mathbb{B}}}

\newcommand{\CO}{{\mathcal{O}}}

\newcommand{\IR}{{\mathbb{R}}}
\newcommand{\IC}{{\mathbb{C}}}

\newcommand{\tol}{{\texttt{tol}}}
\def\a{\alpha}
\def\b{\beta}
\def\d{\delta}

\def\l{\lambda}

\def\s{\sigma}
\def\th{\theta}

\def\d{\delta}

\def\th{\theta}
\newcommand{\tn}[1]{\ \textnormal{#1}\ }
\newcommand{\bmt}{\left[ \begin{array}{ccccccccccccccccccccccccccccccccccccc}}
	\newcommand{\emt}{\end{array}\right]}
\usepackage{booktabs}

\newcommand{\bmtx}{\left[ \begin{array}{c|cccccccccccccccccccccccccccccccccccc}}
	\newcommand{\emtx}{\end{array}\right]}
\begin{document}

\title{A  {well-conditioned} direct PinT  algorithm   for  first- and second-order evolutionary equations}


\author{Jun Liu \and Xiang-Sheng Wang\and    Shu-Lin Wu \and Tao Zhou 
}

\institute{J. Liu \at Department of Mathematics and Statistics, Southern Illinois University Edwardsville, Edwardsville, IL 62026, USA. \email{juliu@siue.edu}\\
\and X. Wang \at Department of Mathematics, University of Louisiana at Lafayette, Lafayette, LA 70503, USA. \email{xswang@louisiana.edu}
\and S. L. Wu  (corresponding author) \at
              School of Mathematics and Statistics, Northeast Normal University, Changchun 130024, China\\
              \email{wushulin84@hotmail.com}           %  \\
              \and
              T. Zhou \at LSEC, Institute of Computational Mathematics and Scientific/Engineering Computing, AMSS, Chinese Academy of Sciences, Beijing, 100190, China. \email{tzhou@lsec.cc.ac.cn}
}

\maketitle

\begin{abstract}
{In this paper, we study a direct parallel-in-time (PinT) algorithm for {first- and second-order  time-dependent differential equations}. We use a second-order boundary value method as the time  integrator. Instead of solving the corresponding all-at-once system iteratively  we  diagonalize  the time discretization matrix $B$, which     yields   a direct parallel implementation across all time \mg{steps}.  A crucial issue of this methodology is  how the condition number (denoted by ${\rm Cond}_2(V)$) of the eigenvector matrix $V$  of  $B$ behaves  as     $n$  grows, where $n$  is the number of time \mg{steps}. A large condition number leads to large roundoff error in the \dg   procedure,  which could  seriously  pollute  the numerical accuracy. Based on a novel  connection between the characteristic equation and the Chebyshev polynomials, we present explicit   formulas  for  $V$ and $V^{-1}$, by which we prove  that  Cond$_2(V)=\CO(n^{2})$.  This implies that the diagonalization process is well-conditioned and the roundoff error   only increases moderately as $n$ grows and thus, compared to other    direct PinT algorithms,  a much  larger $n$  can be used to yield satisfactory parallelism.  {A fast structure-exploiting algorithm is also designed for computing the spectral diagonalization of $B$.}
 Numerical results on parallel machine are given to support our findings, where over 60 times speedup is achieved with 256 cores. }
\end{abstract}

\section{Introduction}\label{sec1}
For time evolutionary problems, parallelization in the time direction is an active  research topic in recent years. This is driven by the fact that in modern supercomputer the number of cores (or threads) grows rapidly year by year, but in many cases one observes that the space parallelization does not bring further speedup even with more cores \cite{falgout2017multigrid}.  When such a saturation  occurs, it is natural to ask whether the time direction can be used for   further speedup or not.  The answer is positive, at least for strongly dissipative problems, for which the widely used parareal algorithm \cite{LMT01} and many other variants (e.g., the  MGRiT algorithm \cite{FF14} and the PFASST algorithm\cite{EM12}) work very well.  However, for wave propagation problems  the performance  of these {representative}   algorithms is   {unsatisfactory}, because {the convergence} rate heavily depends on the dissipativity (see \cite{W17,SRS15} for   discussions).  There are also many efforts toward ameliorating the convergence behavior of the iterative PinT  algorithms via improving  {the} coarse grid correction \cite{DM13,NT20,CH14,FC06,RK12}, but as pointed out in \cite{R18} these modified algorithms   either need   significant additional computation burden (leading to
further degradation of efficiency) or have very limited applicability.
%For wave propagation  problems, the diagonalization-based PinT algorithm  \textit{ParaDiag} \cite{GLW20}  shows promising advantage.

Non-iterative (or direct) PinT algorithms are also proposed in recent years, for which the parallelism  depends on the number of time points only.  Here, we are interested in the PinT algorithm based on the \textit{diagonalization} technique, which was first proposed in 2008 by Maday and   R{\o}nquist  \cite{MR08}.  The idea can be described conveniently for linear ODE system with initial-value condition {(the nonlinear case will be addressed in Section \ref{sec2})}:
\begin{equation}\label{eq1.1}
u'(t)+Au(t)=g(t),
\end{equation}
where $u(0)=u_0\in\IR^m$ is the initial condition, {$A\in\mathbb{R}^{m\times m}$} and $g$ is a known term.  First,  we discretize the temporal derivative by a finite difference scheme (e.g., the backward-Euler method as described {below}) {with a step size $\Delta t$ and uniform time grid points $t_j=j\Delta t, j=0,1,\cdots,n$. Here and hereafter $n$ denotes the number of time points. \mg{Different from   solving these difference equations sequentially one after another},  we formulate them  into an all-at-once fully discrete linear system
\begin{equation}\label{eq1.2}
\mathcal{M}{\bm u}:=\left(B\otimes I_x+I_t\otimes A\right){\bm u}={\bm b},
\end{equation}
where {$\bm u=[u_1^\T,u_2^\T,\cdots,u_n^\T]^\T$ with $u_j\approx u(t_j)$,
$\bm b$ contains the initial condition and right-hand-side information,}
$I_x\in\mathbb{R}^{m\times m}, I_t\in\mathbb{R}^{n\times n}$ are identity matrices and $B\in\mathbb{R}^{n\times n}$ is the time discretization matrix.  Then, {assuming $B$ is diagonalizable,   i.e.,  $B=VDV^{-1}$ with $D=\text{diag}(\lambda_1,\lambda_2,\dots,\lambda_{n})$,  we can} factorize   $\mathcal{M}$ as $$\mathcal{M}=(V\otimes I_x)(D\otimes I_x+I_t\otimes {A})(V^{-1}\otimes I_x).$$
  This leads to  the following {three-step} procedure for directly solving \eqref{eq1.2}:
 \begin{equation}\label{eq1.3}
 \begin{cases}
{\bm g}=(V^{-1}\otimes I_x){\bm b}, &\text{step-(a)},\\
(\lambda_j I_x+A)w_j=g_j, ~j=1,2,\dots, n, &\text{step-(b)},\\
{\bm u}=(V\otimes I_x){\bm w}, &\text{step-(c)},
\end{cases}
\end{equation}
where ${\bm w}=(w_1^\T, w_2^\T, \dots, w_{n}^\T)^\T$ and ${\bm g}=(g_1^\T, g_2^\T, \dots, g_{n}^\T)^\T$. For the first and third steps in \eqref{eq1.3}, we only need to do matrix-vector (or matrix-matrix) multiplications that are parallelizable. The major computational cost is to solve the $n$ linear systems in step-(b), but  these linear systems are completely decoupled and therefore  can be solved in parallel by direct or iterative solvers.

The crucial question is how to efficiently and accurately diagonalize  the    time discretization    matrix $B$.  We mention that the matrix $B$ from standard time discretization may be not diagonalizable. For example, for   the backward-Euler method using a uniform step-size $\Delta t$ {the time discretization  matrix $B$ reads}
\begin{subequations}
\begin{equation}\label{eq1.4a}
B=\frac{1}{\Delta t}\begin{bmatrix}
1 &  & &\\
-1 &1 & &\\
&\ddots &\ddots &\\
& &-1 &1
\end{bmatrix},
\end{equation}
and it is clear that $B$ can not be diagonalized. (For   other time-integrators, e.g., the multistep methods, $B$ is a lower triangular Toeplitz matrix and can not be diagonalized as well.)   {To get a diagonalizable $B$, the strategy in \cite{MR08} is to use distinct step-sizes} $\{\Delta t_j\}_{j=1}^{n}$, which leads to
  \begin{equation}\label{eq1.4b}
B=\begin{bmatrix}
\frac{1}{\Delta t_1}&  & &\\
-\frac{1}{\Delta t_2} &\frac{1}{\Delta t_2} & &\\
&\ddots &\ddots &\\
& &-\frac{1}{\Delta t_{n}} &\frac{1}{\Delta t_{n}}
\end{bmatrix}.
\end{equation}
\end{subequations}
 \mg{Clearly, the matrix  $B$  in \eqref{eq1.4b} has  $n$ distinct eigenvalues and therefore it  is diagonalizable. In practice,  the condition number of the eigenvector matrix $V$   may
be   very large}.   This would be series problem, since a  large condition number  results in large  roundoff error  in the implementation of step-(a) and step-(c) of \eqref{eq1.3} due to floating point operations, which  {could} seriously   pollute  the accuracy of the obtained numerical solution.  This issue was carefully justified by Gander \textit{et al.} in \cite{GH19} and in particular
\begin{equation}\label{eq1.5}
\texttt{roundoff error}=\mathcal{O}(\epsilon {\rm Cond}_2(V)),
\end{equation}
where $\epsilon$ is the machine precision.
 In \cite{GH19}, the authors considered the geometrically increasing step-sizes $\{\Delta t_j=\Delta t_1\tau^{n-j}\}_{j=1}^n$ and with this choice  an explicit \dg of $B$ can be written down,  where  $\tau>1$ is a parameter.  However,  {it is} very difficult to make a good choice of $\tau$:  if $\tau$ gets closer to 1 the matrix $B$ tends to be non-diagonalizable and the condition number of   the eigenvector matrix $V$ becomes  very large; if $\tau$ is far greater than 1 the global discretization error will be an issue, because the step-sizes grows rapidly (exponentially) as $n$ increases. To balance the roundoff error and the discretization error,    numerical results indicate that  $n$ can   be only about 20$\sim$25 (see  the numerical results in Section \ref{sec4.1})   and therefore  the parallelism   is  limited for a large $n$.

{
Here,  we  remove this undesired   restriction on $n$ by using  a hybrid time discretization consisting of a centered finite difference  scheme  for   the first $(n-1)$ time steps   and
  an implicit Euler method for the last step, that is
\begin{equation}\label{eq1.6}
\begin{cases}
\frac{u_{j+1}-u_{j-1}}{2\Delta t}+Au_j=g_j, ~j=1,2,\dots, n-1,\\
\frac{u_{n}-u_{n-1}}{\Delta t}+Au_{n}=g_{n}.
\end{cases}
\end{equation}
Such an implicit  time discretization should not be used in a   time-stepping fashion, due to the serious stability problem. For \eqref{eq1.6}, the   all-at-once system in the form of  \eqref{eq1.2}  is   specified by
\begin{equation}\label{eq1.7}
B=\frac{1}{\Delta t}\begin{bmatrix}
0 &\frac{1}{2} & & & \\
-\frac{1}{2} &0 &\frac{1}{2}  & & \\
  &\ddots &\ddots &\ddots &\\
 &  &-\frac{1}{2}  &0   &\frac{1}{2}\\
 &  &   &-1   &1\\
\end{bmatrix},~ {\bm b}=\begin{bmatrix}\frac{u_0}{2\Delta t}+g_1\\ g_2\\ \vdots\\ g_n\end{bmatrix},~{\bm u}=\begin{bmatrix}u_1\\ u_2\\ \vdots\\ u_n\end{bmatrix},
\end{equation}
where only the initial-value $u_0$ is needed
and all time steps are solved  in one-shot manner. We mention that there are other diagonalization-based  PinT algorithms,  which   use novel  preconditioning tricks to handle the all-at-once system  \eqref{eq1.2} and perform well for large $n$; see, e.g.,  \cite{GW19,LW20,MPW18,LNS18,danieli2021spacetime,caklovic2021parallel,Palitta2021}}. \blue{For example, in \cite{Palitta2021} the author  proposed  a very efficient solution procedure based on  matrix equation formulation that exploiting the extended and rational Krylov subspace projection
techniques for the spatial operator and the circulant-plus-low-rank structure for the time discrete operator.}
These are however iterative algorithms and are not within the {scope} of this paper.
}

The  time discretization \eqref{eq1.6} is not new and according to our best knowledge it was  first proposed in 1985 by Axelsson and Verwer \cite{AV85}, where the authors studied this scheme with the aim of circumventing the well-known Dahlquist-barriers between convergence and stability  which arise in using \eqref{eq1.6} in a time-stepping mode.    In the general  nonlinear case, they proved that the numerical solutions obtained  simultaneously are of uniform second-order accuracy (see Theorem 4 in \cite{AV85}), even though the last step is a first-order scheme.  Numerical results in \cite{AV85} indicate that  the  time discretization \eqref{eq1.6} is suitable for stiff problems in both linear and nonlinear cases. Besides \eqref{eq1.6}, a very similar time discretization   investigated by Fox in 1954 \cite{F54}  and Fox and Mitchell in 1957 \cite{FM57}  appears much earlier, where   instead of the backward-Euler method the authors use the BDF2 method for the last step   in \eqref{eq1.6}:
$${\frac{3u_{n}-4u_{n-1}+u_{n-2}}{2\Delta t}+Au_{n}=g_{n}.}$$
 {For the time discretization \eqref{eq1.6},    the all-at-once system   was carefully justified by Brugnano, Mazzia and   Trigiante  in 1993 \cite{BMT93}, who focus on solving it iteratively by   constructing some effective preconditioner}.   The  implementation of the preconditioner  in \cite{BMT93} relies  on  two operations: a block odd-even cyclic reduction  of $\mathcal{M}$ and  a scaling procedure  for the resulted matrix by its   diagonal blocks. The block cyclic reduction requires matrix-matrix multiplications concerning  $A$   and the scaling requires to invert   $I_x+4\Delta t^2A$ and $I_x+2\Delta tA(I_x+\Delta tA)$. In our opinion, both operations are expensive if $A$ arises from semi-discretizing a PDE in high dimension and/or with {small} mesh sizes.
  Nowadays,   the  hybrid   time discretization \eqref{eq1.6} is a famous example  of the so-called \textit{boundary value methods} (BVMs)  \cite{BT98}, which are widely used in scientific and engineering computing.

Inspired by the pioneering work by Maday and   R{\o}nquist  \cite{MR08}, in this paper we try to solve the all-at-once system \eqref{eq1.2}  directly  (instead of iteratively as in \cite{BMT93}) based on diagonalizing the time discretization matrix $B$ in \eqref{eq1.7} as $B=VDV^{-1}$.    {By discovering a novel connection between the characteristic equation and the Chebyshev  polynomials}, we present  explicit formulas for these three matrices $V$, $V^{-1}$ and $D$. With the given formulas of $V$ and $V^{-1}$, we prove that the condition number of   $V$ satisfies Cond$_2(V)=\CO(n^{2})$  and    this implies that the roundoff error arising from the diagonalization procedure only increases moderately as $n$ grows. Hence, compared to the {algorithm}  in \cite{GH19},   a much  larger $n$  can be used to yield satisfactory parallelism in practice.  We mention that the  spectral  decomposition algorithm developed in this paper is much faster than the benchmarking algorithm   implemented by MATLAB's \texttt{eig} function.

{For second-order  problems
\begin{subequations}
\begin{equation}\label{eq1.8a}
u''(t)+Au(t)=g, ~u(0)=u_0, u'(0)=\tilde{u}_0,
\end{equation}
 we prove in Section \ref{sec2.2} that   {the}  time discretization \eqref{eq1.6} leads to a similar all-at-once system
\begin{equation}\label{eq1.8b}
 ({B^2}\otimes I_x+I_t\otimes A){\bm u}={\bm b},
\end{equation}
 \end{subequations}
 where  ${\bm b}$ is a suitable vector (see Lemma \ref{lem2.1} for details). Thus, the same \dg  of $B$ with {squared eigenvalues}, i.e., $B^2=VD^2V^{-1}$,  can be directly reused and the condition number of the eigenvector matrix is not effected. In other words, there is no essential difference for our proposed algorithms between first-order and second-order problems.     {For both the first-order and the second-order problems, we would like to mention some related fast algorithms in time that may be integrated with our proposed algorithm,
  	such as space-time discretizations\cite{Andreev_2014a,Andreev_2014b}, low-rank approximations \cite{Stoll_2015,kt_2011}, and domain decomposition \cite{Barker_2015}.
  }

The  {remainder of this paper} is organized as follows. In Section \ref{sec2} we  introduce the direct  PinT algorithm for nonlinear problems.   In Section \ref{sec3} we show details of the diagonalization of the time discretization matrix $B$ in \eqref{eq1.7}, {which plays a central role for both the linear and nonlinear cases}. Some numerical results are given in Section \ref{sec4} and we conclude this paper in Section \ref{sec5}.
{The  technical details for estimating Cond$_2(V)$   are given  in Appendix A and a fast algorithm with $\mathcal{O}(n^2)$  complexity for stably  computing $V^{-1}$ is described in Appendix B}.
\section{The PinT algorithm for nonlinear problems}\label{sec2}
In this section, we introduce the time discretization and the  diagonalization-based PinT algorithm for nonlinear problems. We will consider   differential equations with first- and second-order  temporal derivatives separately.  \subsection{First-order problems}\label{sec2.1}
We first consider the following first-order problem
\begin{equation}\label{eq2.1}
u'(t)+f(u(t))=0, u(0)=u_0,
\end{equation}
where $t\in(0, T)$, $u(t)\in\mathbb{R}^m$ and $f: (0, T)\times\mathbb{R}^m\rightarrow\mathbb{R}^m$.
This is an ODE problem, but the algorithm described below is also directly applicable to semi-discretized time-dependent PDEs. For example, \eqref{eq2.1}   corresponds to   the heat {equation}   by letting $f(u)=Au$ with $A\in\mathbb{R}^{m\times m}$ being the discrete  matrix of the negative Laplacian $-\Delta$ by any  discretization (e.g, finite difference or finite element).  Similarly,   the second-order problems considered in subsection \ref{sec2.2}  corresponds to the wave {equation} upon semi-discretization in space.

For \eqref{eq2.1}, similar to  \eqref{eq1.6} the   time discretization scheme is
\begin{equation}\label{eq2.2}
 \begin{cases}
 \frac{u_{j+1}-u_{j-1}}{2\Delta t}+f(u_j)=0, ~j=1,2,\dots, n-1,\\
    \frac{u_{n}-u_{n-1}}{\Delta t}+f(u_{n})=0,
 \end{cases}
\end{equation}
where   the last step is the first-order backward-Euler scheme.
The all-at-once system of  \eqref{eq2.2}  is
 {\begin{equation}\label{eq2.3}
  (B\otimes I_x){\bm u}+F({\bm u})={\bm b},
 \end{equation}
 where $F({\bm u})=(f^\T(u_1), f^\T(u_2),\dots, f^\T(u_n))^\T$ and ${\bm b}=(u_0^\T/(2\Delta t), 0, \dots, 0)^\T$.}
{Applying the standard Newton's iteration   to \eqref{eq2.3} leads to}
   \begin{equation*}
   (B\otimes I_x+ \nabla F({\bm u}^k))({\bm u}^{k+1}-{\bm u}^k)={\bm b}-((B\otimes I_x){\bm u}^k+ F({\bm u}^k)),
 \end{equation*}
 i.e.,
   \begin{equation}\label{eq2.4}
   (B\otimes I_x+ \nabla F({\bm u}^k)){\bm u}^{k+1}={\bm b}+ \left(\nabla F({\bm u}^k){\bm u}^k-F({\bm u}^k)\right),
 \end{equation}
 where $k\geq0$ is the iteration index and $\nabla F({\bm u}^k)={\rm \mathtt{blkdiag}}(\nabla f(u_1^k),   \dots, \nabla f(u_{n}^k))$ {consists of the Jacobian matrix $\nabla f(u_j^k)$ as the $j$-th block}.  {To make the \dg technique  still applicable, we have to replace (or approximate) all the blocks $\{\nabla f(u_j^k)\}$   by a single matrix $A_k$.  Following the interesting idea in \cite{GH17},   we consider the following}
 averaged Jacobian matrix\footnote[1]{{An alternative way of deriving such an aggregated Jacobian matrix is to take the average of unknowns  instead:
 		$
 		 A_k=\nabla f\left (\frac{1}{n}{\sum}_{j=1}^{n} u_j^k\right ),
 		$
 		which is omitted since it shows similar convergence performance in numerical experiments.
 }}
 {$$
A_k:=\frac{1}{n}{\sum}_{j=1}^{n}\nabla f(u_j^k).
 $$}
 Then,   we  get a {simple Kronecker-product} approximation of $\nabla F({\bm u}^k)$ as
 $$
 \nabla F({\bm u}^k)\approx I_t\otimes  {A_k}.
 $$
 By substituting this into \eqref{eq2.4}, we arrive at the simplified Newton iteration  (SNI):
    \begin{equation}\label{eq2.5}
   (B\otimes I_x+ I_t\otimes  {A_k}){\bm u}^{k+1}={\bm b}+ \left((I_t\otimes  {A_k}){\bm u}^k-F({\bm u}^k)\right).
 \end{equation}
 Convergence of SNI is well-known; see, e.g., \cite[Theorem 2.5]{D04} and \cite{OR00}.  The SNI  was also used as an inner iteration for the inexact Uzawa method \cite{NS19} and the Krylov subspace method \cite{LW20}.

 {With the same structure, the Jacobian system \eqref{eq2.5} in each SNI can also be solved parallel in time.}
 If  $B$ is diagonalized  as $B=VDV^{-1}$, we can solve ${\bm u}^{k+1}$ in  \eqref{eq2.5} as
  \begin{equation}\label{eq2.6}
 \begin{cases}
{\bm g}=(V^{-1}\otimes I_x){\bm r}^k, &\text{step-(a)},\\
(\lambda_j I_x+A_k)w_j=g_j, ~j=1,2,\dots, n, &\text{step-(b)},\\
{\bm u}^{k+1}=(V\otimes I_x){\bm w}, &\text{step-(c)},
\end{cases}
\end{equation}
where   ${\bm r}^k={\bm b}+ \left((I_t\otimes {A_k}-F({\bm u}^k)\right)$.  In the linear case, i.e., $f(u)=Au$,  we have $ A_k=A$ and ${\bm r}^k={\bm b}$ and therefore \eqref{eq2.6} reduces to \eqref{eq1.3}. {In the parallel experiments in Section 4 (implemented with MPI and run by Slurm Workload Manager),  $n$ is an integer multiple of the number of processors and the workload is quite evenly distributed for each processor because  the linear systems are of same size.}

 \subsection{Second-order problems}\label{sec2.2}
 We next consider the  following second-order differential equation
 \begin{equation}\label{eq2.7}
u''(t)+f(u(t))=0, u(0)=u_0, u'(0)=\tilde{u}_0, ~t\in(0, T).
\end{equation}
For  discretization we { let $v(t)=u'(t)$ and} make an {\em order-reduction}    by rewriting \eqref{eq2.7} as   {\begin{equation}\label{eq2.8}
w'(t):=\begin{bmatrix}
u(t)\\
v(t)
\end{bmatrix}'=
\begin{bmatrix}
v(t)\\
-f(u(t))
\end{bmatrix}=:g(w),\quad %~u(0)=u_0, v(0)=\tilde{u}_0.
w'(0):=\begin{bmatrix}
	u(0)\\
	v(0)
\end{bmatrix}=
\begin{bmatrix}
	u_0\\
	\tilde{u}_0
\end{bmatrix}.
\end{equation}
}
%Let $w(t)=(u^\T(t), v^\T(t))^\T$ and $g(w)=(v^\T(t),-f^\T(u(t)))^\T$.
Then, similar to \eqref{eq2.2}, the same time discretization scheme leads to
\begin{equation}\label{eq2.9}
 \begin{cases}
 \frac{w_{j+1}-w_{j-1}}{2\Delta t}+g(w_j)=0,~j=1,2,\dots, n-1, \\
   \frac{w_{n}-w_{n-1}}{\Delta t}+g(w_{n})=0.
 \end{cases}
\end{equation}
{Clearly, for \eqref{eq2.9}} the all-at-once system is of the same form as in \eqref{eq2.3} and the \dg   procedure \eqref{eq2.6} is directly applicable.   \mg{However} one can imagine that  the storage requirement for the space variables doubles at each time point  and  this would be undesirable if the second-order problem \eqref{eq2.7} arises from semi-discretizing a PDE in high dimension and/or with {small} mesh sizes.
 {We can avoid this by  representing the all-at-once systems  for ${\bm u}=(u_1,u_2,\dots, u_n)^\T$ only.  \begin{lemma}[all-at-once system for ${\bm u}$]\label{lem2.1}
The vector  ${\bm u}=(u_1^\T, \dots, u_n^\T)^\T$  specified by the time discretization \eqref{eq2.9} satisfies
\begin{equation}\label{eq2.10}
 (B^2\otimes I_x){\bm u}+F({\bm u})={\bm b},
 \end{equation}
 where $B$ is the matrix defined by \eqref{eq1.7} and ${\bm b}=
		 \left(\frac{\tilde{u}_0^\T}{2\Delta t}, -\frac{u_0^\T}{4\Delta t^2} ,  0, \dots, 0\right)^\T$.
 \end{lemma}
 \begin{proof}
 Since $w_j=(u_j^\T, v_j^\T)^\T$, from \eqref{eq2.9} we can represent $\{u_j\}$ and $\{v_j\}$ separately as
 \begin{equation*}
 \begin{split}
&\begin{cases}
 \frac{u_{j+1}-u_{j-1}}{2\Delta t}-v_j=0,~j=1,2,\dots, n-1, \\
   \frac{u_{n}-u_{n-1}}{\Delta t}-v_{n}=0,
 \end{cases}\\
  &\begin{cases}
 \frac{v_{j+1}-v_{j-1}}{2\Delta t}+f(u_j)=0,~j=1,2,\dots, n-1, \\
   \frac{v_{n}-v_{n-1}}{\Delta t}+f(u_{n})=0.
 \end{cases}
 \end{split}
 \end{equation*}
Hence, with the matrix $B$ given by \eqref{eq1.7} we have  	
 \begin{equation}\label{eq2.11}
		(B\otimes I_x){\bm u}-{\bm v}={\bm b}_1, ~(B\otimes I_x){\bm v}+F({\bm u})={\bm b}_2,
	\end{equation}
	 where ${\bm v}=(v_1^\T, \dots, v_n^\T)^\T$, ${\bm b}_1=(\frac{u_0^\T}{2\Delta t}, 0, \dots, 0)^\T$ and
${\bm b}_2=(\frac{\tilde{u}_0^\T}{2\Delta t}, 0, \dots, 0)^\T$.  From the first equation in \eqref{eq2.11} we have
${\bm v}=(B\otimes I_x){\bm u}-{\bm b}_1$ and   substituting this into the second equation  gives  $(B\otimes I_x)^2{\bm u}+F({\bm u})={\bm b}_2+(B\otimes I_x){\bm b}_1$.  A routine calculation yields ${\bm b}_2+(B\otimes I_x){\bm b}_1={\bm b}$ and this together with     $(B\otimes I_x)^2=B^2\otimes I_x$  gives  the desired result  \eqref{eq2.10}.
\end{proof}

\mg{ If $f(u)=Au$, we have $F({\bm u})=(f^\T(u_1), \dots, f^\T(u_n))^\T=(I_t\otimes A){\bm u}$ and thus the all-at-once system \eqref{eq2.10} for ${\bm u}$ becomes   $ (B^2\otimes I_x+I_t\otimes A){\bm u}={\bm b}$, which gives  \eqref{eq1.8b}.}
Clearly, {$B^2$ is  diagonalizable as  $B^2=VD^2 V^{-1}$} given $B=VDV^{-1}$}. Based on this relationship, it is clear that the above PinT algorithm \eqref{eq2.6} is also applicable to \eqref{eq2.11} and the details are  omitted.
Hence, for the diagonalization-based PinT algorithm the computational cost of second-order problems is the same as the first-order ones.
  \section{Diagonalization of the time discretization matrix $B$}\label{sec3}
  For both the linear and nonlinear problems, {it is clear that} the \dg  of $B=VDV^{-1}$ plays a central role in the PinT algorithm. In this section, we will prove that the matrix $B$ is indeed diagonalizable and also give  explicit formulas for   $V$ and  $V^{-1}$. By these formulas, we   {give} an estimate of the 2-norm condition number of $V$, {i.e., Cond$_2(V)=\CO(n^2)$}, which is critical to  control    the roundoff error  in practical computation (cf. \eqref{eq1.5}).

For notational simplicity, we consider the diagonalization of the re-scaled matrix
$\CB= {\Delta t} B$.
Clearly,  {by diagonalizing} $\CB=V \Sigma V^{-1}$    {it holds}
 \[
 B=\frac{1}{\Delta t}\CB=V \left(\frac{1}{\Delta t}\Sigma\right) V^{-1}=VDV^{-1}.
 \]
Define two functions  $$
 T_n(x)=\cos(n\arccos x),~
 U_n(x)=\sin[(n+1)\arccos x]/\sin(\arccos x),
 $$
 {which} are respectively  the $n$-th degree Chebyshev polynomials of {the} first-   and second-kind.
 {In the following theorem  we express
 the eigenvalues and eigenvectors  of $\CB$ through the Chebyshev polynomials.
 Throughout this paper,}  ${\rm i}=\sqrt{-1}$ denotes the imaginary unit.
 \begin{theorem} \label{thm_eigsys}
 	The $n$ eigenvalue{s} of $\CB$ are    {$\l_j={\rm i} x_j$,
	 with $\{x_j\}_{j=1}^n$}  being the $n$ roots of
 	\begin{equation}\label{e-value}
 		U_{n-1}(x)-{\rm i} T_n(x)=0.
 	\end{equation}
 For each $\lambda_j$,
   the corresponding eigenvector $\bm p_j=[p_{j,0}, \cdots,p_{j,n-1}]^\T$ is given as
 	\begin{equation}\label{e-vector}
 		p_{j,k}={\rm i}^kU_k(x_j),~k=0,\cdots,n-1,
 	\end{equation}
    where $p_{j,0}=1$ is assumed for normalization.
 \end{theorem}
\begin{proof}
Let $\l\in\IC$ be an eigenvalue of $\CB$ and $\bm p=[p_0,p_1,\cdots,p_{n-1}]^\T\ne 0$ {the}  corresponding eigenvector.	
By definition we have $\CB\bm p=\l \bm p$, i.e.,
 \begin{equation}
 \begin{cases}
 	\l p_0=\frac{p_1}{2},\\
 	\l p_1=-\frac{p_0}{2}+\frac{p_2}{2},\\
 	\vdots\\
 	\l p_{n-2}=-\frac{p_{n-3}}{2}+\frac{p_{n-1}}{2},\\
 	\l p_{n-1}=-p_{n-2}+p_{n-1}.
 \end{cases}
	\end{equation}
 Obviously, $p_0\neq0$; otherwise, $p_1=\cdots=p_{n-1}=0$. Without loss of generality, we may assume $p_0=1$.
{Clearly,}  $p_k$ is a polynomial of $\l$ with degree $k$. Moreover,
 $p_1=2\l$ and the recursion
 \begin{equation}\label{pk-de}
 	2\l p_{k-1}=p_k-p_{k-2},
 \end{equation}
 holds for $k=2,\cdots,n-1$, and the last equation gives
 \begin{equation}\label{pn}
 	(1-\l)p_{n-1}=p_{n-2}.
 \end{equation}
 Let $\l=\frac{1}{2}(y-\frac{1}{y})={\rm i}\cos\th$ with $y={\rm i}e^{{\rm i}\th}$.
The general solution of the difference equation \eqref{pk-de} is
 \begin{equation}
 	p_k=c_1y^k+c_2(-y)^{-k}.
 \end{equation}
 Making use of the initial conditions $p_0=1$ and $p_1=2\l=y-y^{-1}$, we have
 \begin{align*}
 	c_1+c_2 =1, ~
 	c_1y-c_2y^{-1} =y-y^{-1},
 \end{align*}
 which gives
 	$c_1 ={y\over y+y^{-1}}$ and
 	$c_2 ={y^{-1}\over y+y^{-1}}$.
 Therefore, with {$y={\rm i}e^{{\rm i}\th}$ we get}
 \begin{equation}\label{pk}
 	p_k={y^{k+1}+(-1)^ky^{-(k+1)}\over y+y^{-1}}={{\rm i}^k\sin[(k+1)\th]\over\sin\th},~{k=0,\cdots,n-1}.
 \end{equation}
 In view of $\l={\rm i}\cos\th$, we rewrite \eqref{pn} as
 $$(1-{\rm i}\cos\th){{\rm i}^{n-1}\sin(n\th)\over\sin\th}={{\rm i}^{n-2}\sin[(n-1)\th]\over\sin\th},$$
 which is equivalent to
 \begin{equation}\label{th-eq}
 	{\sin(n\th)\over\sin\th}={\rm i}\cos(n\th).
 \end{equation}
 This is a polynomial equation of $\l={\rm i}\cos\th$ with degree $n$ {because $\sin(n\th)/\sin\th=U_{n-1}(-i\l)$ and $\cos(n\th)=T_n(-i\l)$ are polynomials of $\l$ with degrees $n-1$ and $n$, respectively.}

 Denote $\l={\rm i}x$ with $x=\cos\theta$ (i.e. $\theta=\arccos x$). It follows from \eqref{pk} and \eqref{th-eq} that
 \begin{equation}\label{e-vector2}
 	p_k={\rm i}^kU_k(x),~k=0,1,\dots, n-1,
 \end{equation}
  and
 \begin{equation}\label{e-value2}
 	U_{n-1}(x)-{\rm i}T_n(x)=0.
 \end{equation}
 {The $n$ roots $x_1,x_2,\cdots,x_n$ of   \eqref{e-value2} give the $n$ eigenvalues $\lambda_j={\rm i} x_j$ of $\CB$, and the formula \eqref{e-vector2} evaluated} at each $x_j$ then provides the corresponding eigenvector.\end{proof}

Based on the above Theorem \ref{thm_eigsys},
{we can further  prove that $\CB$ is indeed diagonalizable, since  its  eigenvalues are all distinct}.

\begin{theorem} \label{thm_eigs}
 All $n$ roots of $U_{n-1}(x)-{\rm i}T_n(x)=0$ are simple, complex with negative imaginary parts, and have modulus less than $1+1/\sqrt{2n}$.
Moreover, if $x$ is a root, then so is $-\bar x$.

\end{theorem}	
\begin{proof}
	{From    \eqref{e-value},  it is clear that}  $U_{n-1}(x)-{\rm i}T_n(x)=0$ has no real roots.
	Define $y=x+\sqrt{x^2-1}$ for $x\in\C\setminus[-1,1]$. It holds  $x=\frac{1}{2}(y+\frac{1}{y})$ and $|y|>1$.
	Moreover,
	$$
	T_n(x)=\frac{1}{2}(y^n+y^{-n}),~
	U_{n-1}(x)=\frac{y^n-y^{-n}}{y-y^{-1}}.
	$$
	Thus, if $U_{n-1}(x)-{\rm i}T_n(x)=0$,  we have $(y^n-y^{-n})/(y^n+y^{-n})=(y-y^{-1})/(-2{\rm i})$, which {gives}
	\begin{equation}\label{y-eq}
		y^{2n}={-2{\rm i}+y-y^{-1}\over-2{\rm i}-y+y^{-1}}=-{y^2-2{\rm i}y-1\over y^2+2{\rm i}y-1}=-{(y-{\rm i})^2\over(y+{\rm i})^2}.
	\end{equation}
	Since $|y|>1$, {the above equation implies that} $|y-{\rm i}|>|y+{\rm i}|${, which gives} $\im y<0$.
	Consequently,
	\begin{equation}
		\im x={\im y-\im y/|y|\over 2}<0.
	\end{equation}
	Moreover, it follows from \eqref{y-eq} that
	$$|y|^{2n}={|y-{\rm i}|^2\over|y+{\rm i}|^2}\le{(|y|+1)^2\over(|y|-1)^2}.$$
	Let $y_1=|y|-1>0$. We have
	$$2+y_1\ge y_1(1+y_1)^n\ge y_1(1+ny_1)=y_1+ny_1^2,$$
	which implies  $y_1\le\sqrt{2/n}$.
	{Thus}, $|x|<{|y|+1\over 2}\le 1+{1\over\sqrt{2n}}$. 	If $x$ is a root, then
	$$U_{n-1}(-\bar x)=(-1)^{n-1}\bar U_{n-1}(x)=(-1)^{n-1}(-{\rm i})\bar T_n(x)={\rm i}T_n(-\bar x),$$
	which implies that $-\bar x$ is also a root.
	A simple application of Pythagorean theorem yields
    {
    $$
    T_n^2(x)+(1-x^2)U_{n-1}^2(x)=\cos^2(n\th)+\sin^2(n\th)=1,
    $$
    for $x=\cos\th\in(-1,1)$. Since the left-hand side of the above equation is the sum of two polynomials in $x$, we have for all complex $x$,}
	\begin{equation}
		T_n^2(x)+(1-x^2)U_{n-1}^2(x)=1.
	\end{equation}
	Hence, if $x$ is a root of $U_{n-1}(x)-{\rm i}T_n(x)=0$,  it holds   $x^2T_n^2(x)=1$. 	If $x$ {is} a repeated root, then
	$$2xT_n^2(x)+2x^2T_n(x)T_n'(x)=0.$$
	Since $T_n'(x)=nU_{n-1}(x)$, we have
	$$T_n(x)=-nxU_{n-1}(x)=-{\rm i}nxT_n(x),$$
	which implies   $x={\rm i}/n$ {and this} contradicts to the fact that $\im x<0$.
\end{proof}

By Theorem  \ref{thm_eigs},  the  eigenvectors  of $\CB$  are linearly independent and so $\CB$ indeed is diagonalizable.
Denote the \dg of $\CB$  by  $\CB=V\Sigma V^{-1}$
with $\Sigma=\diag\left(\l_1, \cdots,\l_n\right)$ and
\begin{equation}  \label{Vformula}
 V=[\bm p_1,\bm p_2, \cdots, \bm p_n]= \underbrace{\diag\left({\rm i}^0,{\rm i}^1,\cdots,{\rm i}^{n-1}\right)}_{:=\mathbf{I}}
 \underbrace{\bmt
 	U_0(x_1)&\cdots & U_0(x_n) \\
   	\vdots &\cdots   &\vdots \\
 	U_{n-1}(x_1) &\cdots & U_{n-1}(x_n) \\
 	\emt}_{:=\Phi}=\mathbf{I}\Phi,
\end{equation}
where  {$\{\lambda_j\}_{j=1}^n$ and $\{x_j\}_{j=0}^{n-1}$ are  specified by Theorem \ref{thm_eigsys}.  In \eqref{Vformula},
$\mathbf{I}$ is a unitary matrix and $\Phi$} is a Vandermonde-like matrix \cite{higham2002accuracy} {defined by} the Chebyshev orthogonal polynomials.
{Hence, it holds}
\begin{equation}  \label{eq3.15}
{\rm Cond}_2(V)={\rm Cond}_2(\mathbf{I}\Phi)={\rm Cond}_2(\Phi).
\end{equation}
{The following theorem proves that   ${\rm Cond}_2(V)=\mathcal{O}(n^2)$,
which implies that the roundoff error  from diagonalization procedure only increases moderately as $n$ grows (cf. \eqref{eq1.5}).
Such {a quadratic growth rate} of   ${\rm Cond}_2(V)$}
is crucial to achieve a {satisfactory parallelism}  in time.

\begin{theorem} \label{thmCondU}
	 For $n\ge 8$, it holds
	\begin{equation} \label{condUorder}
	{\rm Cond}_2(V)= \mathcal{O}(n^2).
 	\end{equation}
\end{theorem}
\begin{proof}
{From \eqref{eq3.15}, the proof lies in proving  ${\rm Cond}_2(\Phi)= \mathcal{O}(n^2)$  by using  the Christoffel-Darboux formula and    some  special  properties of relevant orthogonal polynomials.   The details are  quite  technical and hence given  in  Appendix A for better readability.}
\end{proof}
%{The condition number of eigenvector matrix is well-known to be very useful in numerical analysis, but it is also notoriously difficult to estimate  in general, unless some special structure  of the eigenvector matrix can be disclosed, e.g. the Chebyshev-Vandermonde matrix $\Phi$ we encountered here.}
%We mention that the estimate \eqref{condUorder} is a little bit conservative and in practice we find    ${\rm Cond}_2(V)=\CO(n^{1.75})<\CO(n^{2})$ (see Figure \ref{fig3.1} for an illustration), but at the moment we can not rigorously prove such a tight bound yet. {There is a very rich literature of estimating condition numbers of Vandermonde or Vandermonde-like matrices, see e.g. \cite{gautschi1983condition,gautschi2011optimally,kuian2019optimally,brezinski2013walter},
%but the existing results can not be directly applied to our matrix $\Phi$ due to a very different distribution of nodes.}

{An interesting byproduct of Appendix A is    the precise estimate of each individual eigenvalue  of $\mathbb{B}$, which allows us to accurately compute all $n$ different complex eigenvalues   by  Newton's method with $\mathcal{O}(n)$ complexity} (see the following subsection \ref{fasteig} for details).
%\begin{figure}[htp!]
%	\centering
%	\includegraphics[width=0.6\textwidth]{CondV2.eps}
%	%\includegraphics[width=0.45\textwidth]{dt.eps}
%	\caption{In practice, the condition number of the eigenvector matrix $V$ satisfies ${\rm Cond}_2(V)=\CO(n^{1.75})$, which is better than the theoretical bound \eqref{condUorder}. } \label{fig3.1}
%\end{figure}

{\begin{remark}[fast algorithm  for $V^{-1}$]
Making use of the special structure of $\Phi$ (cf.  \eqref{Vformula}), in Appendix B we {give   a}  stable and fast algorithm with complexity  $\mathcal{O}(n^2)$  to  compute  $V^{-1}$. We believe that this algorithm is   of independent interest since it provides a very different idea for inverting  the   Vandermonde-like matrix, which is a well-known ill-conditioned problem and  {a lot of research has been devoted to it},  such as  \cite{higham1988fast,reichel1991chebyshev,calvetti1993fast,gohberg1994fast,gohberg1997fasta,gohberg1997fastb} to name a few.
We present some numerical results in  Section \ref{fasteig} to demonstrate the efficiency of  the proposed    algorithm.
{We remark that some fast inversion algorithms in the literature may not be stable for our Vandermonde-like matrix $V$, mainly due to \mg{its definition} over complex nodes $\{x_i\}_{i=1}^n$. For example, we have tested the fast algorithm given in \cite{gohberg1994fast},
which is very unstable and becomes inaccurate even with $n\ge 32$. Based on our numerical experiments, our proposed algorithm seems to be very stable and it shows $O(n^2)$ complexity, but  a comprehensive comparison with   other fast inversion algorithms deserves further investigation that is beyond our focus.
}
\end{remark}}

\section{Numerical results}\label{sec4}
In this section, we present  some numerical examples to illustrate {the advantage of {the proposed  PinT algorithm}, with respect to numerical accuracy, stable spectral decomposition and   parallel efficiency.    For the first two {subsections}, the results are obtained by   using MATLAB on {a Dell Precision 5820 Tower Workstation with Intel(R) Core(TM) i9-10900X CPU@3.70GHz CPU and 64GB RAM}.
For parallel computation {in subsection \ref{sec4.3}}, we  use a
 parallel computer (SIUE Campus Cluster)  with  10 CPU nodes}   connected via 25-Gigabit per second (Gbps) Ethernet network,
where each node is equipped with
two AMD EPYC 7F52 16-Core Processors at 3.5GHz base clock
and 256GB RAM.
 For the complex-shift  linear systems in step-(b) of the direct PinT algorithm  \eqref{eq1.3}, we use the LU factorization-based   solver provided as \texttt{PCLU} preconditioner in {PETSc \cite{petsc-web-page,petsc-user-ref}}.
%It is possible to solve these systems in step-(b) inexactly with modern Krylov subspace iterative solvers, whose convergence rates are usually highly depend on the systems' condition numbers.
%According to our experience with MATLAB's Parallel Computing Toolbox, the straightforward \texttt{parfor} implementation of step-(b) in our PinT direct solver would not lead to the anticipated speedup.
In parallel examples, let $J(n,s)$  be the measured CPU time  (wall-clock) by using $s$ cores for $n$ time points.  {Following the standard principles   \cite{chopp2019introduction,bueler2020petsc}, we measure}  the parallel speedup  as
\[
\mbox{Speedup (Sp.)}=\frac{J(n,1)}{J(n,s)}.
\]
{The strong and weak scaling efficiency with $s$ cores are computed respectively as}
\[
\mbox{Strong Efficiency (SE)}= \frac{J(n,1)}{s\times J(n,s)}
,~
\mbox{Weak Efficiency (WE)}=\frac{J(2,1)}{J(2 \times s,s)}.
\]
We highlight that the measured parallel speedup and efficiency are affected by many factors, such as {the} computer cluster network setting and how to implement the parallel codes. Hence our parallel  results may largely underestimate the best possible speedup results  with optimized  codes,
{but they do clearly illustrate the practical parallel efficiency of our proposed algorithm.}
\subsection{Accuracy comparison of two direct PinT algorithms}\label{sec4.1}
As   mentioned in Section \ref{sec1},  the  direct PinT algorithm based on the diagonalization technique  was {carefully analyzed} in \cite{GH19}, where the authors used the geometrically increasing step-sizes   to the make the time discretization matrix $B$ diagonalizable. Compared to that  algorithm, the most important advantage of our PinT algorithm lies in the much weaker dependence of the roundoff error (due to diagonalization) on $n$.  The first set of numerical results are devoted to comparing such a dependence for these two algorithms. To this end, we consider the following 1D wave equation
\begin{equation}\label{eq4.1}
u_{tt}-u_{xx}=0, ~u(x,0)=\sin(2\pi x), ~u'(x,0)=0, ~(x,t)\in(-1,1)\times (0,T),
\end{equation}
with  periodic boundary condition $u(-1,t)=u(1,t)$.  Applying the centered finite difference method in space with a uniform mesh $\{x_j=j\Delta x\}_{j=1}^m$ gives a  second-order linear ODE system
\begin{equation}\label{eq4.2}
{\bm u}_h''+A{\bm u}_h=0, ~{\bm u}_h(0)={\bm u}_{0,h}, ~{\bm u}'_h(0)=0, ~t\in (0,T),
\end{equation}
where
$$
A=\frac{1}{\Delta x^2}
\begin{bmatrix}
2 &-1 & & &-1\\
-1 &2 &-1 & &\\
&\ddots &\ddots &\ddots &\\
& &-1 &2 &-1\\
  -1& & &-1 &2
\end{bmatrix},~{\bm u}_{0,h}
=\begin{bmatrix}
\sin(2\pi x_1)\\
\sin(2\pi x_2)\\
\vdots\\
\sin(2\pi x_{m})\\
\end{bmatrix},~\Delta x=\frac{2}{m+1}.
$$

For \eqref{eq4.2}, the diagonalization-based PinT algorithm in \cite{GH19} is based on the Trapezoidal rule  (TR) as   time-integrator, where the step-sizes are fixed by  $\Delta t_j=\Delta t_{n}\tau^{j-n}$  for $j=1,2,\dots, n$   with $\tau>1$ being a constant and $\Delta t_{n}$ being given a {priori}\footnote[1]{The step sizes in \cite{GH19} are     $\Delta t_j=\Delta t_1\tau^{n-j}$ with $\Delta t_1$ being given a {priori}. Here, to control the global discretization error we first fix the last step size $\Delta t_n$ and then specify the step sizes as $\Delta t_j=\Delta t_n\tau^{j-n}$.}.
For reader's convenience,  we briefly explain some details of the algorithm in \cite{GH19}.   By letting  ${\bm w}=[{\bm u}_h^\T, {\bm v}_h^\T]^\T\in\mathbb{R}^{2m}${, then} we can rewrite \eqref{eq4.2} as
 \begin{equation}\label{neq4.3}
 {\bm w}'+\underbrace{\begin{bmatrix} & -I_x\\ A &\end{bmatrix}}_{:=Q}{\bm w}=0, ~{\bm w}(0)={\bm w}_0:=\begin{bmatrix}{\bm u}_{0,h}\\0\end{bmatrix}.
\end{equation}  Let
 \begin{equation}\label{neq4.4}
B_1=\begin{bmatrix}\frac{1}{\Delta t_1} &  & &\\
-\frac{1}{\Delta t_2} &\frac{1}{\Delta t_2} & &\\
 &\ddots &\ddots &\\
& &-\frac{1}{\Delta t_{n}} &\frac{1}{\Delta t_{n}}
\end{bmatrix},~B_2=\frac{1}{2}\begin{bmatrix}1   & & &\\
1 &1 & &\\
 &\ddots &\ddots &\\
& &1 &1
\end{bmatrix}, ~\tilde{\bm b}=\begin{bmatrix}\frac{w_0}{\Delta t_1}-\frac{Qw_0}{2}\\0\\\vdots\\0\end{bmatrix},
 ~{\bm w}=\begin{bmatrix}w_1\\ w_2\\\vdots\\ w_{n}\end{bmatrix}.
\end{equation}
Then, the all-at-once   system  of TR applied to \eqref{eq4.2} is
 \begin{equation}\label{neq4.5}
(B\otimes I_x){\bm w}+(I_t\otimes Q){\bm w}={\bm b}, ~B:=B_2^{-1}B_1, ~{\bm b}:=(B_2^{-1}\otimes I_x)\tilde{\bm b}.
\end{equation}
From \cite{GH19} it holds      $B=VDV^{-1}$ with  $D={\rm diag}(\frac{2}{\Delta t_j})$,  $V=\tilde{V}\tilde{D}$ and
$$
 \tilde{V}=\begin{bmatrix}
 1 & & & &\\
 p_1 &1 & & &\\
  p_2 &p_1 &1 & &\\
 \vdots &\ddots &\ddots &\ddots &\\
 p_{n-1}    &\dots &p_2 &p_1 &1
 \end{bmatrix},~
  \tilde{V}^{-1}=\begin{bmatrix}
 1 & & & &\\
 q_1 &1 & & &\\
  q_2 &q_1 &1 & &\\
 \vdots &\ddots &\ddots &\ddots &\\
q_{n-1}    &\dots &q_2 &q_1 &1
 \end{bmatrix}, ~\tilde{D}={\rm diag}\left(\frac{1}{\sqrt{1+\sum_{l=1}^{n-j}|p_l|^2}}\right),
 $$
 where $p_j= {\prod_{l=1}^j\frac{1+\tau^l}{1-\tau^l}}$ and $q_j=q^{-j}{\prod_{l=1}^j\frac{1+\tau^{-l+2}}{1-\tau^{-l}}}$.
 The diagonal matrix $\tilde{D}$ is used to reduce the condition number of the eigenvector matrix $V$ (if we simply use  $V=\tilde{V}$ the condition number is   larger).  With the above given spectral factorization  of $B$, we can solve the all-at-once system \eqref{neq4.5} via the same diagonalization procedure \eqref{eq1.3}.

Let  $\Delta x=\frac{1}{64}$, $\Delta t_{n}=10^{-2}$ and  $\tau=1.15$.   We let $n$ vary from 4 to 50 and for each $n$  we  implement the diagonalization-based algorithm in \cite{GH19} by using the variable step-sizes. Then,
we calculate the length of the time interval\footnote[2]{Since $\Delta t_j=\Delta t_{n}\tau^{j-n}$ the length of time interval grows as $n$  increases.   }, i.e., $T(\tau, n)=\sum_{j=1}^{n}\Delta t_j$ and implement the   algorithm proposed in this paper by using a uniform step-size    $\Delta t={T(\tau, n)}/{n}$.  Define the global error of numerical solution as
 \begin{equation}\label{eq4.6}
\mbox{global error}=\max_{j=1,2,\dots, n}\|{\bm u}_{j,h}-{\bm u}^{\rm ref}_{j,h}\|_\infty,
\end{equation}
where $\{{\bm u}^{\rm ref}_{j,h}\}$ denotes the reference solution obtained by using the   \texttt{expm} function in MATLAB. That is ${\bm w}_j= \texttt{expm}(-t_jQ){\bm w}_0$ and ${\bm u}^{\rm ref}_{j,h}={\bm w}_j(1:m)$. The sequence $\{{\bm u}_{j,h}\}$  is obtained via three ways:    by the algorithm studied in this paper,  by the algorithm   in \cite{GH19} and   by the {sequential} time-stepping TR using   the variable step-sizes.

In Figure \ref{fig4.1} on the  left, we compare the global error for these three numerical solutions and it is clear that  for the algorithm in \cite{GH19} the quantity $n$ can not be large and the error grows rapidly when $n>25$.
As denoted by the black solid line, the error  of the time-stepping   TR does not change dramatically as $n$ increases and this is because  for each $n$ the last step-size $\Delta t_{n}$ (i.e., the largest step-size) is fixed. For the time-stepping TR, the global error is just the time discretization error.   By comparing  the   dash-dot blue line (with marker `{$\circ$}')  with the black solid line, we can see how the roundoff error affects the global error: when $n$ is small the roundoff error is smaller than the time discretization error and therefore the influence of the roundoff error is invisible, but when $n$ is large (say $n>25$) the roundoff error plays a dominate role and  blows up as $n$ increases.  From \cite{GH19}, we know that such a rapid increase of the roundoff error  is due to the very large condition number of $V$ of the time discretization matrix $B$ in \eqref{neq4.5}.  Indeed, as we can see in Figure \ref{fig4.1} on the   right,  such a condition number  becomes very large
as $n$ grows.  On the contrary,
 the condition number for the new algorithm   only moderately increases  as $n$ grows and it is much smaller.    {Such a {well-conditioned $V$} can be used to explain  the result in Figure \ref{fig4.1} on the left for the new algorithm}:    the global error   never blows up and in fact it continuously decreases when $n\geq6$. {(The small  condition number implies that  the roundoff error is much smaller and thus the global error is dominated by the time discretization error.)}  The decreasing of the global error can be explained as follows. The step-size
$$
\Delta t= \frac{1}{n}{\sum}_{j=1}^{n}\Delta t_{n}\tau^{j-n}=\Delta t_n\frac{1-\tau^{-n}}{n(1-\tau^{-1})}\approx\frac{0.0766}{n} ~(\text{if}~n\geq40)
$$
 decreases as $n$ grows and thus the time discretization error decreases accordingly.
{This error plot in the left of Figure \ref{fig4.1} is not suitable for verifying the second-order of accuracy of our scheme, since the time step-size $\Delta t$ is not small. Such a second-order accuracy will be verified in Table 4.2-4.4.}
  \begin{figure}[h]
\centering
\includegraphics[width=0.49\textwidth]{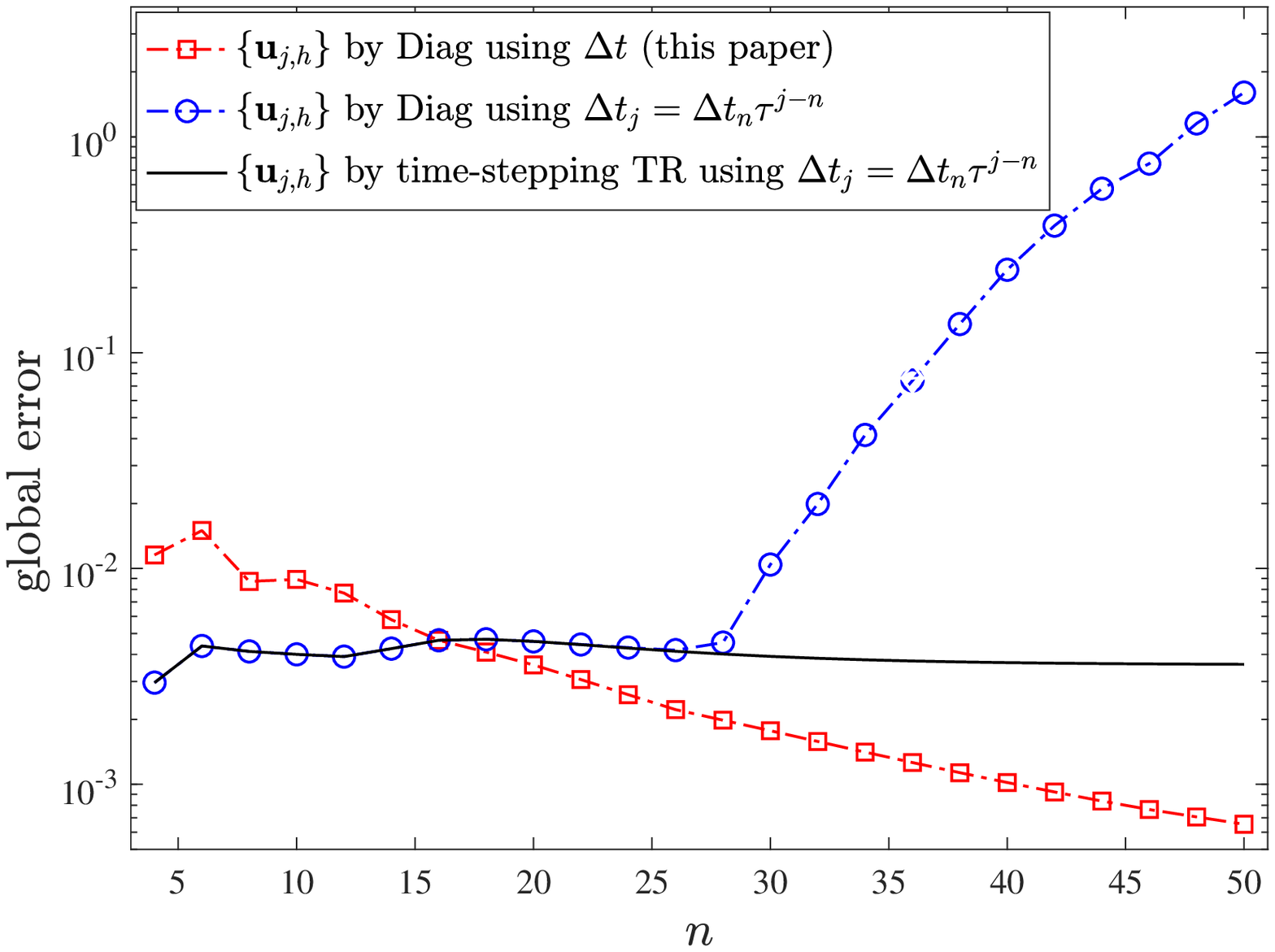} \includegraphics[width=0.49\textwidth]{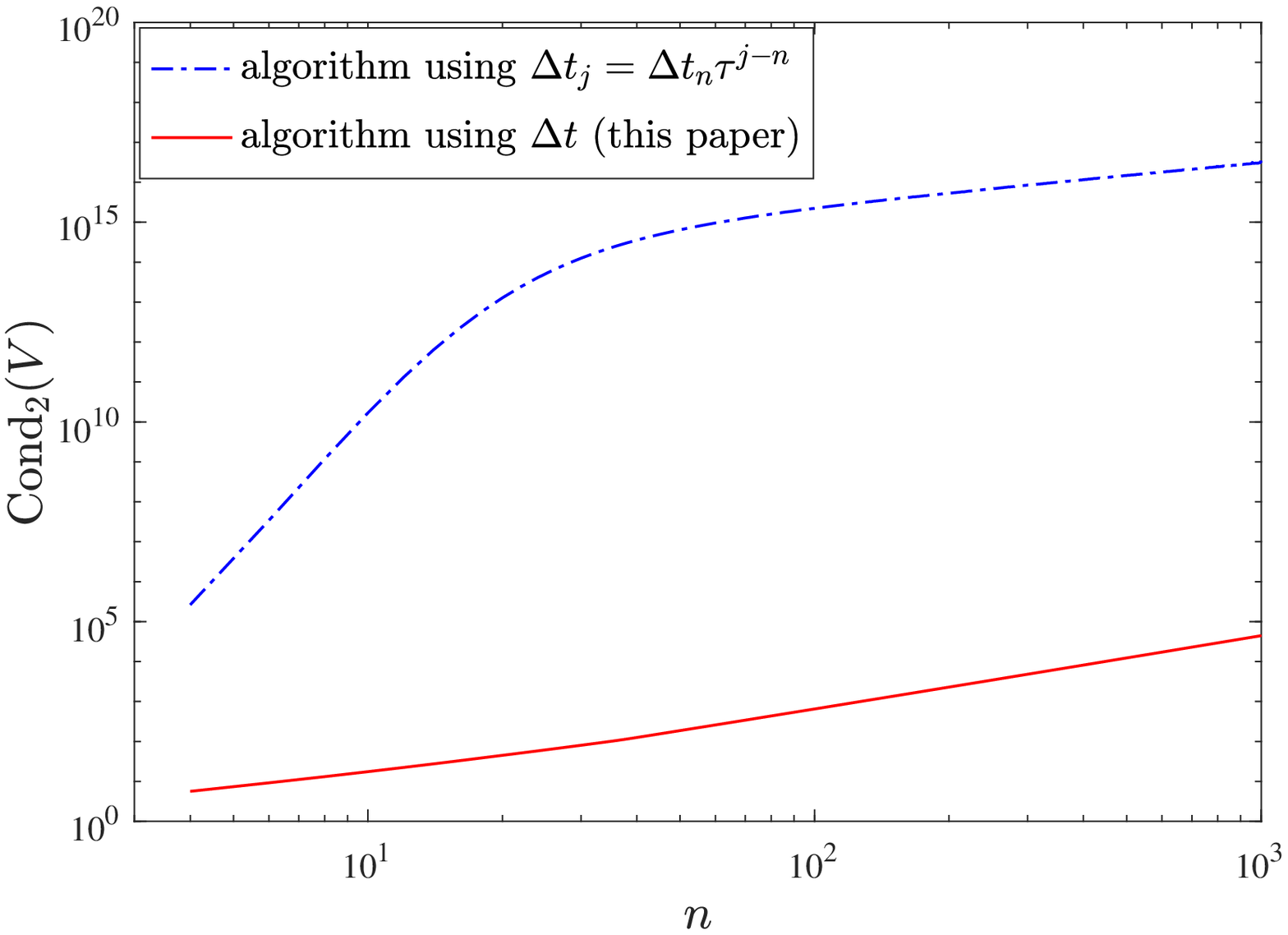}
\caption{Left: the global error for the  new algorithm studied in this paper,  the algorithm  in \cite{GH19} and   the time-stepping TR using the variable step-sizes.  Right: comparison of the condition numbers of the eigenvector matrix $V$ for the two diagonalization-based algorithms.
%Bottom:  with  $T=1$  the global error of the  new algorithm  for a large range of $n$ (the influence of the roundoff error is still invisible).
} \label{fig4.1}
\end{figure}

%Overall, the results shown in Figure \ref{fig4.1} reflect that the new diagonalization-based algorithm has obvious advantage with respect to roundoff error, compared to the algorithm in \cite{GH19}. For the new algorithm, the roundoff error  only weakly depends on $n$ and this confirms our analysis   about the condition number of the eigenvector matrix very well.
%
%

\subsection{Fast spectral decomposition of $B$.} \label{fasteig}
%In Appendix B, we proposed a fast algorithm for computing $V^{-1}$ with complexity $\mathcal{O}(n^2)$. Here, we provide
The spectral decomposition of the time discretization matrix $ B=VD V^{-1}$ is important  in our PinT algorithm.   The eigenvalue  {$\lambda_j$  can be computed  by  Newton's method  as described below.
 Based on  Theorem  \ref{thm_eigsys} (cf.  equation \eqref{th-eq}), it holds $\lambda_j={\rm i} \cos(\theta_j)$, where $\theta_j$ is the $j$-th root of}
 $$\rho(\theta):=\sin(n\theta)-{\rm i}\cos(n\theta)\sin\theta=0.$$
Applying   Newton's  iteration to the nonlinear equation $\rho(\theta)=0$ of single variable $\theta$ leads to
 \begin{equation}\label{eigNewton}
\theta_j^{(l+1)}=\theta_j^{(l)}-\frac{\rho(\theta_j^{(l)})}{\rho'(\theta_j^{(l)})},~ l=0,1,2,\cdots.
\end{equation}
{Such a Newton method runs $n$ loops for the  $n$ eigenvalues $\lambda_j$'s. The maximal iteration  number over all the $n$ eigenvalues   is  almost  constant   and therefore the complexity of Newton's iteration \eqref{eigNewton}   for computing  all the   eigenvalues is of $\CO(n)$, which is significantly faster than the standard QR algorithm with $\CO(n^3)$ complexity as used by MATLAB¡¯s highly optimized built-in  function \texttt{eig}. However, it is rather difficult to choose the $n$ initial guesses $\{\theta_j^{(0)}\}_{j=1}^n$. If these initial guesses are not properly chosen, the $n$  iterates $\{\theta_j^{(l)}\}_{j=1}^n$ converge to $\tilde{n}$ different values with $\tilde{n}<n$, i.e.,  not all the eigenvalues are {found}\footnote[1]{From Theorem \ref{thm_eigs} all the $n$ eigenvalues of $B$ are different}. By Lemma \ref{lem-zeros} in Appendix A, we suggest  using}
\[
\theta_j^{(0)}=\frac{1}{2}\left(\frac{j\pi}{n}+\frac{j\pi}{n+1}\right)+\frac{{\rm i}}{n},~ j=1,2,\cdots, n,
\]
 {by which the iterates of \eqref{eigNewton}  converge to the $n$ different eigenvalues correctly}.

For $V^{-1}$, we also proposed a fast algorithm with complexity $\mathcal{O}(n^2)$ in Appendix B, which is of independent interest in the area of numerical methods  for Vandermonde-like matrices. {The advantage of explicitly constructing the inverse matrix  $V^{-1}$ is to increase the parallel efficiency of step-(a) by reducing communication cost.}
Let $B=V_{\mathtt{eig}}D_{\mathtt{eig}}V^{-1}_{\mathtt{eig}} $
and $B=V_{\mathtt{fast}}D_{\mathtt{fast}}V_{\mathtt{fast}}^{-1} $
be the   spectral decomposition  of $B$ by the \texttt{eig} function and our fast algorithm (implemented with MATLAB), respectively.
Define the {maximal}  relative differences
\[
\eta_\mathtt{fast}:=\frac{\|D_{\mathtt{eig}}-D_{\mathtt{fast}}\|_F}{\|D_{\mathtt{eig}}\|_F},
\]
and
\[
\omega_{\mathtt{eig}}:=\frac{\|B-V_{\mathtt{eig}}D_{\mathtt{eig}}V_{\mathtt{eig}}^{-1}\|_F}{\|B\|_F},~
\omega_\mathtt{fast}:=\frac{\|B -V_{\mathtt{fast}}D_{\mathtt{fast}}V_{\mathtt{fast}}^{-1}\|_F}{\|B\|_F}.
\]
(For $\eta_\mathtt{fast}$ the   eigenvalues  are sorted in the same order.)
In Table \ref{Ex0Tab1eig}, we
show  CPU time  (in seconds)  for the spectral decomposition using the \texttt{eig} function in MATLAB and our fast algorithm.  The combined CPU time is  estimated by the timing functions
\texttt{tic/toc} in MATLAB.   Besides, we also show  the   computational time  for the \texttt{eig} function (for computing $V_{\mathtt{eig}}$ and $\Sigma_{\mathtt{eig}}$) and the {\texttt{mrdivide} (i.e. `/') function (for computing $D_{\mathtt{eig}}V_{\mathtt{eig}}^{-1}$ with the syntax $D_{\mathtt{eig}}/V_{\mathtt{eig}}$)} in MATLAB and our proposed fast spectral decomposition algorithm,
where the column `Iter' denotes the  number of Newton iterations required to reach the tolerance $\tol=10^{-10}$.  The CPU time  of our fast algorithm shows $O(n^2)$ growth, which is significantly less  than that of the   \texttt{eig} function (with $O(n^3)$ growth). In particular,   for $n=8192$  we observed more than 25 times
speedup. The eigenvalues and eigenvectors computed by these two methods   are essentially the same, if we take   into account the effects of roundoff and discretization errors.
{In particular, our proposed fast algorithm for the computation of $V^{-1}$ involves solving the real pentadiagonal linear system (\ref{Sb=v}) and $n$ complex
	tridiagonal sparse linear systems (\ref{Ta=f}), which seems to deliver noticeable degraded approximation accuracy (i.e., larger $\omega_\mathtt{fast}$) mainly due to more round-off errors.
Nevertheless, the achieved accuracy is sufficiently high in view of the second-order accurate discretization errors in space and time.
}
\begin{table}[H]   
	\centering
	\caption{Comparison of \texttt{eig+{mrdivide}} solver and our fast spectral decomposition algorithm}
	
	\begin{tabular}{|c||c|c||c|c|c||c||c|c|c|c|c}
		\hline
		&\multicolumn{2}{c|}{MATLAB's \texttt{eig+{mrdivide}}}&\multicolumn{4}{c|}{Our fast algorithm}\\
		\hline
		$n$&  CPU&$\omega_\mathtt{eig}$&Iter &CPU & $\omega_\mathtt{fast}$& $\eta_\mathtt{fast}$\\
		\hline
% 32&	 0.008 &	  5.65e-15&	 7 &	 0.036 &	 6.07e-14&	 1.04e-15 \\
64&	 0.002 &	  1.59e-14&	 7 &	 0.004 &	 3.61e-13&	 2.67e-15 \\
128&	 0.011 &	  7.77e-14&	 7 &	 0.006 &	 1.06e-12&	 2.77e-15 \\
256&	 0.056 &	  1.89e-13&	 8 &	 0.019 &	 1.10e-11&	 4.55e-15 \\
512&	 0.277 &	  9.69e-13&	 8 &	 0.073 &	 5.30e-11&	 8.89e-15 \\
1024&	 1.107 &	  3.85e-12&	 9 &	 0.301 &	 2.04e-10&	 2.63e-14 \\
2048&	 6.741 &	  1.01e-11&	 9 &	 1.206 &	 5.12e-10&	 1.25e-13 \\
4096&	 60.257 &	  4.02e-11&	 10 &	 5.054 &	 6.75e-09&	 5.16e-13 \\
8192&	 606.045 &	  2.25e-10&	 10 &	 23.402 &	 2.85e-08&	 4.07e-13 \\
		\hline
	\end{tabular}
	\label{Ex0Tab1eig}
\end{table}
 
\subsection{Parallel Experiments}\label{sec4.3}
In this subsection, we provide a series of parallel simulation results to validate the speedup and parallel efficiency of our proposed direct PinT algorithm.

{\bf Example-1}.
In this example we consider a  {2D} heat equation with homogeneous Dirichlet boundary condition  defined on a  square domain $\Omega=(0,\pi)^2$:
\begin{equation}\label{heatPDE}
	\begin{cases}
		u_{t}(x,y,t)- \Delta u(x,y,t) =r(x,y,t), &\tn{in} \Omega\times(0,T),\\
		u(x,y,t)=0,&\tn{on} \partial\Omega\times(0, T), \\
		u(x,y,0)=u_0(x,y),&\tn{in} \Omega,
	\end{cases}
\end{equation}
where
$u_0(x,y)=\sin(x)\sin(y)$ and  $r(x,y,t)=\sin(x)\sin(y)e^{-t}$. {The  exact solution of this problem is}    $u(x,y,t)=\sin(x)\sin(y)e^{-t}$.
{Approximating  $\Delta$ by a  centered finite   difference scheme with a uniform mesh step size {$h$} in both $x$ and $y$ directions  gives the following}  ODE system:
\[
\bm u_h'(t)-\Delta_h \bm u_h(t)=\bm r_h(t),~
\bm u_h(0)=\bm u_{0,h},
\]
where $\Delta_h\in \IR^{{m\times m}}$ is the 5-point stencil Laplacian matrix, $\bm u_h$, $\bm r_h$, $\bm u_{0,h}$ denotes the finite difference approximation to the corresponding   $u$, $r$, $u_0$ over the {$m$} interior spatial grid points.
{In} Table \ref{Ex1Tab1}, we show  the approximation errors (measured by the  $\infty$-norm) and the strong and weak scaling results of our direct PinT solver, where the spatial mesh size is  {$h=\frac{1}{513}$ (i.e., ${m=512^2}$)} and the number of cores ranges from 1 to 256.
The {approximation} errors  in weak scaling results show a second-order accuracy in time before dominated by the discretization errors in space.
Both strong and weak scaling efficiency are very promising up to 32 cores. But when the core number $s\geq64$, we see an obvious drop of the parallel efficiency. This  is mainly due to the slow {interconnection} between the nodes (each node contains 32 cores).
Using a fast, low-latency {interconnection} (e.g., the InfiniBand networking based on remote direct memory access technology) would greatly {further improve the   parallel efficiency}.
{For $n=512$ , from the strong scaling CPU column we observe that  the system can be solved  within 20 seconds via our direct PinT algorithm using  256 cores}, rather than over 20 mins using a single core.
\begin{table}[ht]
	\centering
	\caption{Scaling and error results of example 1: a heat PDE ($T=2$ with ${m=512^2}$)}
	
	\begin{tabular}{|c||c|c||c|c|c||c|c|c|c||c|c|c|c|c|cc}
		\hline
		Core\#&\multicolumn{5}{c|}{strong scaling} & \multicolumn{4}{c|}{weak scaling} \\ \hline
		$s$&	$n$&  Error& CPU &Sp.&SE & $n$&  Error& CPU&WE\\
		\hline
 1      & 512   & 2.23e-06         & 1318.8  &  1.0   & 100.0\%     &   2   & 7.93e-02      & 5.4 & 100.0\%             \\
 2      & 512   & 2.23e-06         & 667.8   &  2.0   & 98.7\%     &   4   & 1.19e-02      & 5.4  & 100.0\%             \\
 4      & 512   & 2.23e-06         & 346.4   &  3.8   & 95.2\%     &   8   & 3.22e-03      & 5.4  & 100.0\%             \\
 8      & 512   & 2.23e-06         & 173.0   &  7.6   & 95.3\%     &  16   & 8.26e-04      & 5.5  & 98.2\%           \\
 16      & 512   & 2.23e-06         & 90.7   &  14.5  & 90.9\%     &  32   & 2.09e-04      & 5.8  & 93.1\%            \\
 32      & 512   & 2.23e-06         & 51.1   &  25.8  & 80.7\%     &  64   & 5.28e-05      & 6.6  & 81.8\%            \\
 64      & 512   & 2.23e-06         & 32.0   &  41.2  & 64.4\%     & 128   & 1.37e-05      & 8.3  & 65.1\%            \\
 128     & 512   & 2.23e-06         & 23.0   &  57.3  & 44.8\%     & 256   & 4.25e-06      & 12.0 & 45.0\%            \\
 256     & 512   & 2.23e-06         & 19.4   &  68.0  & 26.6\%     & 512   & 2.23e-06      & 19.6 & 27.6\%            \\
		\hline
		
	\end{tabular}
	\label{Ex1Tab1}
\end{table}

 {\bf  Example-2}.
{We next} consider a  linear wave equation with homogeneous Dirichlet boundary condition  defined on a 2D square domain $\Omega=(0,1)^2$:
\begin{equation}\label{wavePDE}
	\begin{cases}
		u_{tt}(x,y,t)- \Delta u(x,y,t) =r(x,y,t), &\tn{in} \Omega\times(0,T),\\
		u(x,y,t)=0,&\tn{on} \partial\Omega\times(0, T), \\
		u(x,y,0)=u_0(x,y),&\tn{in} \Omega,\\
		u_t(x,y,0)=\bar u_0(x,y),&\tn{in} \Omega,
	\end{cases}
\end{equation}
{with the following data}
\begin{equation*}
\begin{split}
&u_0(x,y)=0, ~ \bar u_0(x,y)=2\pi  x(x-1)y(y-1),\\
&r(x,y,t)=-4\pi^2x(x-1)y(y-1)\sin(2\pi t)-2\sin(2\pi t)(x(x-1)+y(y-1)).
\end{split}
\end{equation*}
The  exact solution of this problem is  $u(x,y,t)=x(x-1)y(y-1)\sin(2\pi t).$
Using the same notations  {in {\bf Example-1}}, we obtain a second-order   ODE system:
\[
\bm u_h''(t)-\Delta_h \bm u_h(t)=\bm r_h(t),~
\bm u_h(0)=\bm u_{0,h}, ~ \bm u_h'(0)=\bm{\bar u}_{0,h},
\]
where $\bm{\bar u}_{0,h}$ denotes the finite difference approximation {to   $\bar u_0$ over the  spatial grid points. Then,    we show in
Table  \ref{Ex2Tab1}  the approximation errors,  the strong and weak scaling results.   The parallel efficiency  is very similar to   that in Table \ref{Ex1Tab1}}.
Since our PinT algorithm is based on the same spectral decomposition  $B=VDV^{-1}$, the computational cost of solving {the above} second-order  problem  is essentially the same as the {first-order  problem in {\bf Example-1}}. This is a desirable advantage  over those iterative algorithms {(e.g, parareal and MGRiT)},  whose convergence rates are usually much slower for handling  hyperbolic problems.
\begin{table}[ht]
	\centering
	\caption{Scaling and error results of example 2: a wave PDE ($T=2$ with ${m=512^2}$)}
	
	\begin{tabular}{|c||c|c||c|c|c||c|c|c|c||c|c|c|c|c|cc}
		\hline
		Core\#&\multicolumn{5}{c|}{Strong scaling} & \multicolumn{4}{c|}{Weak scaling} \\ \hline
		$s$&	$n$&  Error& CPU &Sp.&SE & $n$&  Error& CPU&WE\\
		\hline
		1       & 512   & 7.88e-05     &1328.6   &1.0     &   100.0\%               &   2   & 9.19e-03        & 5.4 &  100.0\%   \\
		2        & 512   & 7.88e-05    &676.3    &2.0     &   98.2\%               &   4   & 2.21e-02        & 5.4  &  100.0\%  \\
		4        & 512   & 7.88e-05    &332.6    &4.0     &   99.9\%               &   8   & 3.16e-01        & 5.5  &  100.0\% \\
		8        & 512   & 7.88e-05    &172.6    &7.7     &   96.2\%               &  16   & 1.33e-01        & 5.7  &  100.0\% \\
		16      & 512   & 7.88e-05     &91.2     &14.6    &   91.0\%                &  32   & 2.30e-02        & 6.0  & 94.8\%  \\
		32      & 512   & 7.88e-05     &51.7     &25.7    &   80.3\%                &  64   & 5.21e-03        & 7.1  & 82.1\%  \\
		64      & 512   & 7.88e-05     &31.2     &42.6    &   66.5\%                & 128   & 1.27e-03        & 9.5  & 67.9\%  \\
		128     & 512   & 7.88e-05     &23.2     &57.3    &   44.7\%                & 256   & 3.16e-04        & 14.8 & 46.6\%  \\
		256     & 512   & 7.88e-05     &20.3     &65.4    &   25.6\%                & 512   & 7.88e-05        & 27.4 & 28.2\%  \\
		\hline

	\end{tabular}
	\label{Ex2Tab1}
\end{table}

{\bf Example-3}.
{At last},  we consider a semi-linear parabolic equation with homogeneous Dirichlet boundary condition  defined on a 2D square domain $\Omega=(-1,1)^2$:
\begin{equation}\label{heatPDE_nonlin}
	\begin{cases}
		u_{t}(x,y,t)- \Delta u(x,y,t)+f(u) =r(x,y,t), &\tn{in} \Omega\times(0,T),\\
		u(x,y,t)=0,&\tn{on} \partial\Omega\times(0, T), \\
		u(x,y,0)=u_0(x,y),&\tn{in} \Omega,
	\end{cases}
\end{equation}
where
\begin{equation*}
	\begin{split}
&f(u)=u^3-u, ~ u_0(x,y)=(x^2-1)(y^2-1),\\
&r(x,y,t)=-2(x^2-1)(y^2-1)e^{-t}+(x^2-1)^3(y^2-1)^3e^{-3t}-2e^{-t}((x^2-1)+(y^2-1)).
\end{split}
\end{equation*}
This problem has the exact solution    $u(x,y,t)=(x^2-1)(y^2-1)e^{-t}$.
{By  the centered finite difference method in space, we get a} nonlinear ODE system
\begin{equation}\label{eq4.11}
\bm u_h'(t)-\Delta_h \bm u_h(t)+f(\bm u_h(t))=\bm r_h(t),~
\bm u_h(0)=\bm u_{0,h}.
\end{equation}
This particular type of nonlinear function $f(u)=u^3-u$ was widely used in literature, e.g.  the Schl\"{o}gl model in \cite{buchholz2013optimal,Stefan2021}.
{We solve \eqref{eq4.11} by the nonlinear PinT algorithm described in Section \ref{sec2.1}, for which the simplified Newton iteration starts from a zero   initial guess} {and stops whenever the relative residual {norm}  is smaller than the  tolerance $10^{-8}$ (smaller than the level of discretization errors).}

{In Table  \ref{Ex3Tab1}, we   show} the approximation errors and the strong and weak scaling results of the PinT algorithm, where the required number of SNI (listed in the column `SNI')  shows an anticipated mesh-independent convergence rate. {Compared to}  the linear examples, we see that  the parallel efficiency  {becomes  lower, especially  when the core number $s>32$}. {This is  mainly because of the} communication cost in distributing the averaged block-diagonal Jacobian matrices and dispatching the residual vectors during the   Newton iterations.
Our  codes may be further optimized {to achieve}   better parallel efficiency, which is however beyond the scope of {the} current paper.

\begin{table}[H]
	\centering
	\caption{Scaling and error results of example 3: a semi-linear parabolic PDE ($T=2$ with ${m=256^2}$)}
	
	\begin{tabular}{|c||c|c|c|c|c|c||c|c|c||c|c||c|c|c|}
		\hline
	Core\#&\multicolumn{6}{c|}{Strong scaling} & \multicolumn{5}{c|}{Weak scaling} \\ \hline
	$s$&	$n$&  Error&SNI& CPU&Sp.&SE & $n$&  Error&SNI&  CPU&WE\\
	\hline
 1	 & 512 	& 6.36e-07	 	&9	& 1514.0  &1.0    &  100.0\%      &   2 	& 4.40e-01	 	 &11&  6.6& 100.0\%  \\
2	 & 512 	& 6.36e-07	 	&9	& 770.4  & 2.0   &   98.3\%    &   4 	& 7.64e-03	 	 &9	&  5.4    & 122.2\%  \\
4	 & 512 	& 6.36e-07	 	&9	& 400.6  & 3.8   &   94.5\%    &   8 	& 2.33e-03	 	 &11&  6.8    & 97.1\% \\
8	 & 512 	& 6.36e-07	 	&9	& 217.2  & 7.0   &   87.1\%    &  16 	& 6.38e-04	 	 &9	&  5.9    & 111.9\%  \\
16	 & 512 	& 6.36e-07	 	&9	& 126.9  & 11.9   &  74.6\%      &  32 	& 1.63e-04	 	 &9	&  7.0    & 94.3\% \\
32	 & 512 	& 6.36e-07	 	&9	& 84.7   & 17.9   &  55.9\%      &  64 	& 4.07e-05	 	 &9	&  9.4    & 70.2\% \\
64	 & 512 	& 6.36e-07	 	&9	& 67.6   & 22.4   &  35.0\%      & 128 	& 1.02e-05	 	 &9	&  28.8   & 22.9\% \\
128	 & 512 	& 6.36e-07	 	&9	& 60.6   & 25.0   &  19.5\%      & 256 	& 2.55e-06	 	 &9	&  29.8   & 22.1\% \\
256	 & 512 	& 6.36e-07	 	&9	& 60.8   & 24.9   &  9.7\%     & 512 	& 6.36e-07	 	 &9	&  61.6   & 10.7\% \\
	\hline
		
	\end{tabular}
	\label{Ex3Tab1}
\end{table}
 
\section{Conclusions}\label{sec5}
In this paper we developed and analyzed a diagonalization-based direct (or non-iterative) PinT  algorithm  for  first-  and second-order evolutionary problems. The algorithm is based on a second-order boundary value method as the time integrator and the diagonalization of the time discretization matrix.  Explicit formulas for the \dg \mg{are} given and \mg{used to} prove that  {the} condition number of the eigenvector matrix is of order $\CO(n^2)$. The quadratic growth of the condition number with respect to $n$   guarantees that the proposed algorithm {is well-conditioned  and therefore}  can be used to handle a larger number of  time points, which is more practical than the algorithm by Gander {\em et al.} \cite{GH19}. For implementation, we need to compute the inverse of the eigenvector matrix, for which we give a fast algorithm with complexity $\CO(n^2)$ by  exploiting its special structure. Numerical results   indicate  that the proposed direct PinT algorithm has promising advantages   with respect to  roundoff errors and  parallel speedup.

\section*{Acknowledgement} The authors are very grateful to the two anonymous referees for
their careful reading of the original manuscript and their valuable suggestions, which greatly improved the quality of this paper. {Shu-Lin Wu is supported
by the National Natural Science Foundation of China (NSFC) (No. 12171080).}

{\section*{Declarations of Conflicting interests}
The authors declared no   conflicts of interest with
respect to the research, authorship, and/or publication of this
article.}

\appendix
\section*{Appendix A: estimate the condition number of $V$.}

\newcommand{\thechapter}{A}
\setcounter{section}{0}
\setcounter{theorem}{0}
\renewcommand{\thetheorem}{\thechapter.\arabic{theorem}}
\renewcommand{\thelemma}{\thechapter.\arabic{lemma}}
\renewcommand{\theequation}{\thechapter.\arabic{equation}}
\setcounter{figure}{0}
\setcounter{table}{0}
\setcounter{equation}{0}

{The proof of   Theorem \ref{thmCondU} is based on the following lemmas.   Recall the following definitions:  ${\rm i}=\sqrt{-1}$ is the imaginary unit and  $$
	T_n(x)=\cos(n\arccos x),~
	U_n(x)=\sin[(n+1)\arccos x]/\sin(\arccos x),
	$$   are the $n$-th degree Chebyshev polynomials of first   and second kind,  respectively  .}
 {The following lemma provides some nice and frequently used properties of the zeros of the polynomial equation $U_{n-1}(x)-{\rm i}T_n(x)=0$.}
\begin{lemma}\label{lem-zeros}
	The  zeros  of $U_{n-1}(x)-{\rm i}T_n(x)=0$  can be arranged as $x_1, \dots,  x_n$  such that for each $j=1,\dots,n$,  $x_j=\cos(\theta_j)=\cos(\a_j+{\rm i}\b_j)$ with    $\a_j=(j\pi-a_j)/n$ and  $\b_j=b_j/n$, where    $a_j\in(0,\pi)$ and $b_j>0$ satisfy the following equations
	\begin{equation}\label{x-a-b}
		|x_j|={\cosh\b_j\over\cosh b_j}={\sin\a_j\over\sinh b_j}={\sinh\b_j\over\sin a_j}={1\over\sqrt{\cos^2a_j+\sinh^2b_j}},~~\cos a_j\sinh\b_j=\sin a_j\cos\a_j.
	\end{equation}
	Moreover, we have the  symmetric relations
	\begin{equation}\label{ab-symmetry}
		a_j+a_{n+1-j}=\a_j+\a_{n+1-j}=\pi,~~b_{n+1-j}=b_j,~~\b_{n+1-j}=\b_j,~~j=1,\cdots,n,
	\end{equation}
	and the monotone properties (with $m=\lfloor n/2\rfloor$ being the largest integer less than or equal to $n/2$)
	\begin{equation}\label{ab-monotone}
		0<a_1<\cdots<a_m<\pi/2<a_{n+1-m}<\cdots<a_n<\pi,~~
		0<b_1<\cdots<b_m,
	\end{equation}
	and the inequalities
	\begin{equation}\label{ab-inequality}
		b_j>{1\over n},~~a_j<{j\pi\over n+1}<\a_j<{j\pi\over n},~~j=1,\cdots,m.
	\end{equation}
	If $n=2m+1$, then $b_{m+1}>1/2$. If $n=2m$, then $b_m>1/2$.
\end{lemma}
\begin{proof}
	Let $\bar b>0$ be the unique positive root of the equation $\sinh b\sinh(b/n)=1$.
	It is easily seen that the function
	$$f(b):=n\arcsin[\tanh b\cosh(b/n)]+\arcsin[\tanh(b/n)\cosh b]$$
	is strictly increasing on $[0,\bar b]$ with $f(0)=0$ and $f(\bar b)=(n+1)\pi/2$.
	For each positive index $j\le(n+1)/2$, there exists a unique $b_j\in(0,\bar b]$ such that $f(b_j)=j\pi$.
	Define
	\begin{equation*}
		\b_j=b_j/n, ~a_j=\arcsin[\tanh(b_j/n)\cosh b_j],~
		\a_j=(j\pi-a_j)/n=\arcsin[\tanh b_j\cosh(b_j/n)].
	\end{equation*}
	A simple calculation shows that $x_j=\cos(\a_j+{\rm i}\b_j)$ is a root of $U_{n-1}(x)-{\rm i}T_n(x)=0$.
	Moreover, \eqref{x-a-b} holds for $1\le j\le(n+1)/2$.
	For $(n+1)/2\le j\le n$, define
	\begin{equation*}
		b_j=b_{n+1-j},~\b_j=b_j/n=\b_{n+1-j},~a_j=\pi-a_{n+1-j},~
		\a_j=(j\pi-a_j)/n=\pi-\a_{n+1-j}.
	\end{equation*}
	We also obtain \eqref{x-a-b} and $U_{n-1}(x_j)-{\rm i}T_n(x_j)=0$  with $x_j=\cos(\a_j+{\rm i}\b_j)$.
	
	The symmetric properties \eqref{ab-symmetry} follows immediately from the above construction.
	The monotonicity of $f(b)$ on $[0,\bar b]$ and \eqref{ab-symmetry} imply the monotonicity of $a_j$ and $b_j$ in \eqref{ab-monotone}.
	In view of $f(1/n)>\pi$, we obtain $b_j>1/n$. Note from \eqref{x-a-b} that
	$$
	\tan a_j/\tan\a_j=\sinh\b_j/\sin\a_j=\tanh\b_j/\tanh b_j<1.
	$$
	Thus, we have $a_j<\a_j$. This together with $\a_j=(j\pi-a_j)/n$ implies $a_j<j\pi/(n+1)<\a_j$, and  {then \eqref{ab-inequality} follows}.
	Finally, for $n=2m+1$ it holds $b_{m+1}=\bar b>1/2$, because $\sinh(1/2)\sinh[1/(2n)]<1=\sinh\bar b\sinh(\bar b/n)$. For $n=2m$, it holds  $\cos\a_m=\sin(a_m/n)<a_m/n<\pi/(2n)$.  	  In view of \eqref{x-a-b} and $n\ge2$, we have
	$$1=(\cos^2\a_m+\sinh^2\b_m)(\cos^2a_m+\sinh^2b_m)<(0.7+\sinh^2b_m/4)(1+\sinh^2b_m),$$
	which implies that $b_m>1/2.$
	This completes the proof.
\end{proof}

In the following, we always assume that the zeros $x_j$ (as well as $a_j$, $b_j$, $\a_j$ and $\b_j$) are ordered as in Lemma \ref{lem-zeros}. {We denote by $\bar x_j$  the conjugate of $x_j$}.
{The following lemma gives some \mg{sharp bounds  on the modulus of the zeros}, which will be used in the proof of Lemma} \ref{thm-U}.
\begin{lemma}\label{lem-xj}
	Assume $n\ge3$. For any $j=1,\cdots,n$, we have
	\begin{equation}\label{A5}
		|x_j|>{\ln n\over2n}, ~~
		{1\over|x_j^2(\bar x_j-x_j)|}<n^3.
	\end{equation}
\end{lemma}
\begin{proof}
	By symmetry, we only need to consider the case $j\le(n+1)/2$.
	Assume to the contrary that $|x_j|\le(\ln n)/(2n)$ for some $j\le(n+1)/2$.
	Let $\s=(n+1)/2-j\ge0$. We claim $\s<1/2+(\ln n)/4$.
	Otherwise, we have $j\le n/2-(\ln n)/4$, $\a_j<j\pi/n\le \pi/2-(\pi\ln n)/(4n)$, and
	consequently, $|x_j|>\cos\a_j>\sin[(\pi\ln n)/(4n)]>(\ln n)/(2n)$,  which is a contradiction.
	Hence, {it holds}
	$$
	\s<1/2+(\ln n)/4, ~j>n/2-(\ln n)/4.
	$$
	It then follows that $\a_j>j\pi/(n+1)>\pi/2-\pi(\ln n+2)/(4n+4)$.
	Thus, $\cos\a_j<\sin[\pi(\ln n+2)/(4n+4)]<\pi(\ln n+2)/(4n+4)$.
	Since $\sinh\b_j<|x_j|\le(\ln n)/(2n)$, we have $b_j<n\sinh\b_j<(\ln n)/2$ and $\sinh b_j<e^{b_j}/2<\sqrt n/2$.
	Consequently,
	\begin{align*}
		1=(\cos^2\a_j+\sinh^2\b_j)(\cos^2a_j+\sinh^2b_j)
		<\left({\pi^2(\ln n+2)^2\over 16(n+1)^2}+{(\ln n)^2\over 4n^2}\right)(1+{n\over 4}),
	\end{align*}
	which is a contradiction  again. This proves the first inequality in \eqref{A5}.
	
	Next, we note that $|\bar x_j-x_j|=2|\im x_j|=2\sin\a_j\sinh\b_j$ and
	$$|x_j^2(\bar x_j-x_j)|=2(\cos^2\a_j+\sinh^2\b_j)\sin\a_j\sinh\b_j.$$
	If $\cos^2\a_j\ge 0.4$,  {from $\sin\a_j>2\a_j/\pi>2/(n+1)$ and $\sinh\b_j>\b_j>1/n^2$} we have
	$${1\over|x_j^2(\bar x_j-x_j)|}<{n^2(n+1)\over1.6}<n^3.$$
	If $\cos^2\a_j<0.4$,   it follows from $n\ge3$ and \eqref{x-a-b} that
	$$1=(\cos^2\a_j+\sinh^2\b_j)(\cos^2a_j+\sinh^2b_j)
	<(0.4+\sinh^2b_j/9)(1+\sinh^2b_j),$$
	which implies that $b_j>0.87$ and $\sinh\b_j>\b_j=b_j/n>0.87/n$.
	We then have
	$${1\over|x_j^2(\bar x_j-x_j)|}<{1\over2\sqrt{1-0.4}\sinh^3\b_j}<{n^3\over 2\sqrt{0.6}(0.87)^3}<n^3.$$
	Coupling the above two cases yields the second inequality of  \eqref{A5}.
\end{proof}
{The following lemma will be used to estimate $\|\Phi\|_2$.}
\begin{lemma}\label{thm-U}
	For any $j=1,\cdots,n$, {it holds}
	\begin{equation}
		{1\over|x_j|}{\sum}_{k=1}^n\left[{1\over|x_k(\bar x_j-x_k)|}+{1\over2|x_k|}\right]
		=\CO(n^3).
	\end{equation}
\end{lemma}
\begin{proof}
	By symmetry,  we assume  $j\le\frac{n+1}{2}$.
	If $k<\frac{n}{2}$, then $\re x_{n+1-k}<0<\re x_j$. Thus,
	$$|\bar x_j-x_{n+1-k}|>|x_{n+1-k}|=|x_k|>\cos\a_k>\sin{(n/2-k)\pi\over n}>{n-2k\over n},$$
	and
	\begin{equation}\label{x-j+}
		\sum_{k>n/2+1}{1\over|x_k(\bar x_j-x_k)|}
		<\sum_{k<n/2}{1\over|x_{n+1-k}(\bar x_j-x_{n+1-k})|}
		<\sum_{k<n/2}{n^2\over(n-2k)^2}<{\pi^2n^2\over6}.
	\end{equation}
	If $k<n/2$ and $k\neq j$, it holds
	\begin{align*}
		|\bar x_j-x_k|>2\sin{\a_j+\a_k\over2}\sin{|\a_j-\a_k|\over2}
		>{2(j+k-1)(|j-k|-1/2)\over n^2}.
	\end{align*}
	Therefore,
	\begin{align*}
		&{\sum}_{k<n/2,k\neq j}{1\over|x_k(\bar x_j-x_k)|}
		<{\sum}_{k<n/2,k\neq j}{n^3\over2(n-2k)(j+k-1)(|j-k|-1/2)}=\mathcal{O}(n^2).
	\end{align*}
	Finally, since $|x_k(\bar x_j-x_k)|>|\im x_k|^2=\sin^2\a_k\sinh^2\b_k$, it is easy to estimate
	\begin{equation}\label{x-m}
		{\sum}_{n/2\le k\le n/2+1}{1\over|x_k(\bar x_j-x_k)|}
		=\mathcal{O}(n^2).
	\end{equation}
	A combination of the above estimates and Lemma \ref{lem-xj} gives the desired result. \end{proof}
For each $s=1,\cdots,n$, we denote $\th_s=s\pi/(n+1)$ and $y_s=\cos\th_s$. Let
\begin{equation}\label{Lj}
	L_j(x)=\prod_{1\le k\le n,k\neq j}{x-x_k\over x_j-x_k}={U_{n-1}(x)-iT_n(x)\over(x-x_j)[U_{n-1}'(x_j)-iT_n'(x_j)]},
\end{equation}
be the Lagrange interpolation polynomials such that $L_j(x_k)=\d_{jk}$, where   $j=1,\cdots,n$.
{The following lemma will be used to estimate $\|\Phi^{-1}\|_2$.}
\begin{lemma}\label{thm-W}
	For any $j=1,\cdots,n$, we have
	\begin{equation}
		{\sum}_{s=1}^n(1-y_s^2)|L_j(y_s)|{\sum}_{k=1}^n|L_k(y_s)|=\mathcal{O}(n^2).
	\end{equation}
\end{lemma}
\begin{proof}
	A routine calculation gives
	\begin{equation}\label{Lj-ys}
		|L_j(y_s)|=\left|{(-1)^{s-1}(1+{\rm i}y_s)(1-x_j^2)\over(y_s-x_j)({\rm i}-nx_j)x_jT_n(x_j)}\right|
		={|1-x_j^2|\sqrt{1+y_s^2}\over|(y_s-x_j)({\rm i}-nx_j)|}.
	\end{equation}
	Note that $|1-x_j^2|=|\sin(\a_j+{\rm i}\b_j)|^2=\sin^2\a_j+\sinh^2\b_j$
	and $|x_j|>\max\{|\cos\a_j|,\sinh\b_j\}$. Since $y_s$ is real with $y_s^2<1$ and $\im x_j<0$, we have $|y_s-x_j|\ge |\im x_j|=\sin\a_j\sinh\b_j$,
	and
	\begin{equation}\label{Lj-bound}
		|L_j(y_s)|<{\sqrt2|1-x_j^2|\over n|x_j(y_s-x_j)|}<{\sqrt2\sin^2\a_j\over n|x_j(y_s-x_j)|}
		+{\sqrt2\over n\sin\a_j}<{\sqrt2\sin^2\a_j\over n|x_j(y_s-x_j)|}+\sqrt2,
	\end{equation}	
	where we {have   used}   the inequality $\sin\a_j>\sin[\pi/(n+1)]>2/(n+1)>1/n$.
	Another application of $|y_s-x_j|>\sin\a_j\sinh\b_j$ yields
	$$|L_j(y_s)|-\sqrt2<{\sqrt 2\sin\a_j\over n|x_j|\sinh\b_j}<{\sqrt 2\over n|x_j|\b_j}.$$
	If $\cos^2\a_j\le1/2$, then \eqref{x-a-b} implies that
	$1<(1/2+\sinh^2b_j)(1+\sinh^2b_j).$
	Hence, $b_j>0.4$ and $|x_j|>\sinh\b_j>\b_j>0.4/n$.
	If $\cos^2\a_j>1/2$, then $|x_j|>|\cos\a_j|>1/\sqrt2$ and $\b_j=b_j/n>1/n^2$.
	In either case, we have
	\begin{equation}\label{Lj-On}
		|L_j(y_s)|-\sqrt2=\CO(n),~~1\le j, s\le n.
	\end{equation}	
	{We next} estimate the sum
	$\sum_{k=1}^n|L_k(y_s)|$.
	By symmetry, we assume without loss of generality that $s\le(n+1)/2$.
	If $k<m=\lfloor n/2\rfloor$ and $k\neq s$, then it follows from Lemma \ref{lem-zeros} and \eqref{Lj-bound} that
	\begin{align*}
		|y_s-x_k|>2\sin{\a_k+\th_s\over2}\sin{|\a_k-\th_s|\over2}
		>{2(\a_k+\th_s)|\a_k-\th_s|\over\pi^2}>{2(k+s)(|k-s|-1/2)\over(n+1)^2},
	\end{align*}
	and
	\begin{align}\label{s-new}
		&{\sum}_{k=1,k\neq s}^{m-1}(\sqrt2|L_k(y_s)|-2)
	 	<{\sum}_{k=1,k\neq s}^{m-1}{(n+1)\pi^2(k+1/2)^2\over n(n+1-2k)(k+s)(|k-s|-1/2)}=\CO(n).
	\end{align}
	If $k>n+1-m$, then $\re x_k<0\le y_s$ and $|y_s-x_k|>|y_s-x_{n+1-k}|$.
	Moreover, $|{\rm i}-nx_k|=|{\rm i}-nx_{n+1-k}|$.
	It then follows from \eqref{Lj-ys} that $|L_k(y_s)|<|L_{n+1-k}(y_s)|$.
	This together with \eqref{Lj-On} and \eqref{s-new} implies that
	\begin{equation}\label{s-est}
		{\sum}_{k=1}^n(|L_k(y_s)|-\sqrt2)=\CO(n).
	\end{equation}	
	Finally, we want to estimate $\sum_{s=1}^n(1-y_s^2)|L_j(y_s)|$.
	Since $|L_j(y_s)|=|L_{n+1-j}(y_{n+1-s})|$, it suffices to consider the case $j\le(n+1)/2$; namely, $\a_j\le\pi/2$.
	For $1\le s, j\le (n+1)/2$ with $s\neq j$, we have
	\begin{equation*}
		|y_s-x_j|>2\sin{|\th_s-\a_j|\over2}\sin{\th_s+\a_j\over2}
		>{2(\th_s+\a_j)|\th_s-\a_j|\over\pi^2}>{2(j+s)(|j-s|-1/2)\over(n+1)^2},
	\end{equation*}
	and
	$1-y_s^2=\sin^2\th_s<\th_s^2=s^2\pi^2/(n+1)^2$.
	It then follows from \eqref{Lj-bound} that
	\begin{align*}
		(1-y_s^2)(|L_j(y_s)|-\sqrt2)<{\sqrt2\pi^2s^2\over(\ln n)(j+s)(|j-s|-1/2)}.
	\end{align*}
	By a routine calculation, we obtain from the above inequality and \eqref{Lj-On} that
	\begin{equation}
		{\sum}_{s\le(n+1)/2}(1-y_s^2)[|L_j(y_s)|-\sqrt2]=\CO(n).
	\end{equation}
	If $s\ge(n+1)/2$, then $y_s<0<\re x_j$ and $|y_s-x_j|>|y_s-x_{n+1-j}|$.
	It follows from \eqref{Lj-ys} that $|L_j(y_s)|<|L_j(y_{n+1-s})|$.
	On account of $y_s=-y_{n+1-s}$, we obtain
	\begin{equation}\label{j-est}
		{\sum}_{s=1}^n(1-y_s^2)|L_j(y_s)|<
		2{\sum}_{s\le(n+1)/2}(1-y_s^2)[|L_j(y_s)|-\sqrt2]=\CO(n).
	\end{equation}
	Coupling \eqref{s-est} and \eqref{j-est} gives the desired estimate. 	
\end{proof}
\subsection*{\textit{Proof of Theorem \ref{thmCondU}}.}
\begin{proof}
Denote $s_k:=x_kT_n(x_k)$. It is readily seen that
$s_k^2=1$, $U_{n-1}(x_k)={\rm i}s_k/x_k$ and $U_n(x_k)=(1+{\rm i}x_k)s_k/x_k$.
Recall from \eqref{Vformula}) that $\Phi$ is the main component of the eigenvector matrix $V$ of $\mathbb{B}$.
It then follows from {the} Christoffel-Darboux formula that
$$
(\Phi^*\Phi)_{jk}={\sum}_{l=0}^{n-1}U_l(\bar x_j)U_l(x_k)={U_n(\bar x_j)U_{n-1}(x_k)-U_{n-1}(\bar x_j)U_n(x_k)\over2(\bar x_j-x_k)}
={s_js_k(2{\rm i}+\bar x_j-x_k)\over2\bar x_jx_k(\bar x_j-x_k)},
$$
which together with Lemma \ref{thm-U} implies
\begin{align*}
	{\sum}_{k=1}^n|(\Phi^*\Phi)_{jk}|\le{1\over|x_j|}{\sum}_{k=1}^n\left[{1\over|x_k(\bar x_j-x_k)|}+{1\over2|x_k|}\right]=\CO(n^3),
\end{align*}
and
\begin{equation}\label{U2}
	\|\Phi\|_2=\sqrt{\rho(\Phi^*\Phi)}\le\sqrt{\|\Phi^*\Phi\|_\infty}=\CO(n^{3/2}).
\end{equation}
Let $W=(w_{jk})_{j,k=1}^n=\Phi^{-1}$. We obtain from orthogonality and Gaussian quadrature formula that
\begin{equation} \label{w_jk}
	w_{jk}={2\over\pi}\int_{-1}^1L_j(x)U_{k-1}(x)\sqrt{1-x^2}dx={\sum}_{s=1}^n{2(1-y_s^2)\over n+1}L_j(y_s)U_{k-1}(y_s),
\end{equation}
where $L_j(x)$ are the Lagrange interpolation polynomials given by  \eqref{Lj}.
A simple calculation yields
\begin{align*}
	(WW^*)_{jk}={\sum}_{s=1}^n{2(1-y_s^2)\over n+1}L_j(y_s)\bar L_k(y_s),
\end{align*}
which together with Lemma \ref{thm-W} implies
\begin{equation}\label{W2}
	\|W\|_2=\sqrt{\rho(WW^*)}\le\sqrt{|WW^*|_1}=\CO(n^{1/2}).
\end{equation}
Coupling \eqref{U2} and \eqref{W2} gives
$
{\rm Cond}_2(V)={\rm Cond}_2(\Phi)=\|\Phi\|_2\|W\|_2=\mathcal{O}(n^2).
$
\end{proof}
\section*{Appendix B: A fast $\mathcal{O}(n^2)$ algorithm for computing {$V^{-1}$}.}
{From \eqref{Vformula}, the eigenvector matrix $V$ of   $\mathbb{B}$ satisfies $V={\bf I}\Phi$ with ${\bf I}=\diag\left({\rm i}^0,{\rm i}^1,\cdots,{\rm i}^{n-1}\right)$.  In the diagonalization procedure \eqref{eq1.3}, we need to compute $V^{-1}=\Phi^{-1}{\bf I}^{-1}$ and  the major computation is to get  $W=\Phi^{-1}$.   In this appendix, we present a fast and stable $\mathcal{O}(n^2)$ algorithm for computing $W$ accurately. }

\renewcommand{\thechapter}{B}
\setcounter{section}{0}
\setcounter{theorem}{0}
\renewcommand{\thetheorem}{\thechapter.\arabic{theorem}}
\renewcommand{\thelemma}{\Alph{section}\arabic{lemma}}
\renewcommand{\theequation}{\thechapter.\arabic{equation}}
\setcounter{figure}{0}
\setcounter{table}{0}
\setcounter{equation}{0}

A simple application of the recurrence relation $2yU_j(y)=U_{j+1}(y)+U_{j-1}(y)$ gives
\begin{equation}\label{Uk-DE}
	{\small\begin{split}
		4y_s^2U_{k-1}(y_s)&=2y_s[U_{k-2}(y_s)+U_k(y_s)] =
		\begin{cases}
			U_{k-3}(y_s)+2U_{k-1}(y_s)+U_{k+1}(y_s),&~2\le k\le n-1,\\
			U_{k-1}(y_s)+U_{k+1}(y_s),&~k=1,\\
			U_{k-3}(y_s)+U_{k-1}(y_s),&~k=n.
		\end{cases}
	\end{split}}
\end{equation}
It then follows from (\ref{w_jk}) that
\begin{align}
	2w_{jk}=&{1\over n+1}{\sum}_{s=1}^n4L_j(y_s)U_{k-1}(y_s)-{1\over n+1}{\sum}_{s=1}^n4y_s^2L_j(y_s)U_{k-1}(y_s)
	\notag\\=&\begin{cases}
		2\psi_{j,k}-\psi_{j,k-2}-\psi_{j,k+2},&~2\le k\le n-1,\\
		3\psi_{j,k}-\psi_{j,k-2},&~k=1,\\
		3\psi_{j,k}-\psi_{j,k+2},&~k=n,
	\end{cases}
\end{align}
where $\psi_{j,k}={1\over n+1}{\sum}_{s=1}^nL_j(y_s)U_{k-1}(y_s)$.
{Since $U_n(y_s)=0$, we have $\psi_{j,n+1}=0$ for $j=1,2,\cdots,n$.}
Define
\begin{equation}\label{b-k}
	p_n(x)=U_{n-1}(x)-{\rm i}T_n(x), ~b_k={1\over n+1}{\sum}_{s=1}^np_n(y_s)U_{k-1}(y_s).
\end{equation}
Recall from \eqref{Lj} that $p_n(y_s)=p_n'(x_j)(y_s-x_j)L_j(y_s)$. Therefore,
\begin{align}\label{bk-ak2}
	{2b_k\over p_n'(x_j)}=&{1\over n+1}{\sum}_{s=1}^n2(y_s-x_j)L_j(y_s)U_{k-1}(y_s)
	=\psi_{j,k-1}+\psi_{j,k+1}-2x_j\psi_{j,k}.
\end{align}
To evaluate $b_k$, we  investigate the integral of $p_n(x)U_{k-1}(x)\sqrt{1-x^2}$ on $[-1,1]$.
On account of \eqref{Uk-DE} and \eqref{b-k}, we obtain from the Gaussian quadrature formula that
\begin{align}
	{4\over\pi}\int_{-1}^1p_n(x)U_{k-1}(x)\sqrt{1-x^2}dx
	=&{1\over n+1}\sum_{s=1}^n4(1-y_s^2)p_n(y_s)U_{k-1}(y_s)
	\notag\\=&\begin{cases}
		2b_k-b_{k-2}-b_{k+2},&~2\le k\le n-1,\\
		3b_k-b_{k+2},&~k=1,\\
		3b_k-b_{k-2},&~k=n.
	\end{cases}
\end{align}
On the other hand, it follows from a direct computation based on orthogonality that
\begin{equation}
	{4\over\pi}\int_{-1}^1p_n(x)U_{k-1}(x)\sqrt{1-x^2}dx=2{\d_{n,k}}+{\rm i}\d_{k,n-1}.
\end{equation}
Coupling the above two equations yields a sparse pentadiagonal linear system
\begin{equation}\label{Sb=v}
	S_n \mathbf b:=	\begin{pmatrix}
		3&0&-1\\
		0&2&0&-1\\
		-1&0&2&0&-1\\
		&\ddots&\ddots&\ddots&\ddots&\ddots\\
		&&-1&0&2&0&-1\\
		&&&-1&0&2&0\\
		&&&&-1&0&3
	\end{pmatrix}
	\begin{pmatrix}
		b_1\\b_2\\b_3\\\vdots\\b_{n-2}\\b_{n-1}\\b_n
	\end{pmatrix}
	=\begin{pmatrix}
		0\\0\\0\\\vdots\\0\\{\rm i}\\2
	\end{pmatrix}.
\end{equation}
Let $\Psi=[\psi_{jk}]$. The whole fast inversion algorithm for computing $W=\Phi^{-1}$ is given as the following three steps.\\

\begin{enumerate}
	\item[Step-1:] solve $\mathbf b=[b_1,\cdots,b_n]^\T$ from \eqref{Sb=v}, which costs $\CO(n)$ operations by the fast Thomas algorithm.\\

	\item[Step-2:] Based on the fact $\psi_{j,n+1}=0$, the $j$-th row ${\bm \psi}_j=[\psi_{j,1},\cdots,\psi_{j,n}]$ of $\Psi$ can be solved from a sequence of sparse tridiagonal linear systems (for each $j=1,2,\cdots,n$)
	\begin{equation}\label{Ta=f}
		G_j \mathbf{\bm\psi}_j^\T:=   \mbox{Tridiag}\{1,-2x_j,1\} {\bm \psi}_j^\T
		= \frac{2}{p'_n(x_j)}\mathbf{b},
	\end{equation}
	which in total also costs $\mathcal{O}(n^2)$ operations based on the fast Thomas algorithm for each system.\\
	
	\item[Step-3:] $W=\frac{1}{2}\Psi {S_n}$, which also needs $\mathcal{O}(n^2)$ operations since {$S_n$} is a sparse matrix.
\end{enumerate}

In summary, the dense complex matrix $W=\Phi^{-1}$ can be computed with  $\mathcal{O}(n^2)$ complexity.

 {\tiny
 \bibliographystyle{siam} %siam
\bibliography{DirectPinT}
}
\end{document}